\documentclass{amsart}
\usepackage{amssymb}
\usepackage{mathrsfs}
\usepackage{stmaryrd}
\usepackage{bbm}
\usepackage{oldgerm}
\usepackage[english]{babel}
\usepackage[T1]{fontenc}
\usepackage[latin1]{inputenc}
\usepackage[all]{xy}
\usepackage{hyperref}

\newtheorem{theorem}[subsection]{Theorem}
\newtheorem{proposition}[subsection]{Proposition}
\newtheorem{corollary}[subsection]{Corollary}
\newtheorem{lemma}[subsection]{Lemma}

\theoremstyle{definition}

\newtheorem{remark}[subsection]{Remark}

\newtheorem{example}[subsection]{Example}
\numberwithin{equation}{subsection}

\usepackage{amsmath}
\usepackage{amsthm}
\usepackage{mathrsfs}
\usepackage{amsfonts}

\begin{document}

\title {$\ell$-adic Tautological Systems}
\thanks{We thank Quentin Guignard and Haoyu Hu for helpful discussions. 
The main part of this work was done while Lei Fu visited the Center of Mathematical Sciences and Applications (CMSA) at Harvard University
in 2015. He would like to thank CMSA for the hospitality. The research of Lei Fu is supported by NSFC 11531008.}

\author{Lei Fu, An Huang, Bong Lian, Shing-Tung Yau, Dingxin Zhang, Xinwen Zhu}
\address{Yau Mathematical Sciences Center, Tsinghua University, Beijing 100084, P. R. China}
\email{leifu@math.tsinghua.edu.cn}
\address{Department of Mathematics, Brandeis University, Waltham, MA 02454, USA}
\email{anhuang@brandeis.edu}
\address{Department of Mathematics, Brandeis University, Waltham, MA 02454, USA}
\email{lian@brandeis.edu}
\address{Department of Mathematics, Harvard University, Cambridge, MA 02138, USA}
\email{yau@math.harvard.edu}
\address{Yau Mathematical Sciences Center, Tsinghua University, Beijing 100084, China}
\email{zhangdingxin03@gmail.com}
\address{Department of Mathematics, California Institute of Technology, Pasadena, CA 91125, USA}
\email{xzhu@caltech.edu}

\date{}

\begin{abstract}
Tautological systems was introduced in \cite{LY} as the system of differential equations satisfied by period integrals 
of hyperplane sections of some complex projective homogenous varieties. 
We introduce the $\ell$-adic tautological systems for the case where
the ground field is of characteristic $p$. 


\end{abstract}

\maketitle

\section* {Introduction}

Let $G$ be an algebraic group over $\mathbb C$,  $\mathfrak g$ the Lie algebra of $G$, 
$\beta: \mathfrak g\to \mathbb C$ a character,  
$V$ a finite dimensional complex representation of $G$, and $\hat X$ a closed subscheme of $V$
invariant under the action of $G$. For any $\theta\in\mathfrak g$, 
let $L_\theta$ be the vector field on $V$ defined by
$$L_\theta(x)=\frac{d}{dt}\Big|_{t=0}\Big(\exp(t\theta)\cdot x\Big)$$ for any $x\in V$. Similarly, we can define a vector field 
$L_\theta$ on the dual space $V^\vee$ provided with the contragredient action by $G$. 
Fix a basis $\{x_1,\ldots, x_N\}$ on $V$. Let $\{\xi_1,\ldots, \xi_N\}$ be the dual basis on the dual space $V^\vee$, 
and let $I(\hat X)$ be the ideal of of $\mathbb C[\xi_1, \ldots, \xi_N]$ consisting of functions vanishing on $\hat X$.  
The \emph{tautological system} $\tau(G,\beta,V, \hat X)$ is  the following system of differential equations for functions
$f(x_1, \ldots, x_N)$ on $V^\vee$:
\begin{eqnarray*}
&& L_\theta f=\beta(\theta) f,\\
&& g\Big(\frac{\partial}{\partial x_1},\ldots, \frac{\partial}{\partial x_N}\Big) f=0,
\end{eqnarray*}
where the first equation is taken for all $\theta\in\mathfrak g$, and the second equation is taken for all $g(\xi_1,\ldots, \xi_N)\in I(\hat X)$. 
In \cite{LY}, it is proved that under some conditions, the tautological systems are satisfied by period integrals 
of hyperplane sections of some complex projective homogenous varieties. 
The $D$-module corresponding to this systems is introduced in \cite{H}.  
As $f(x_1, \ldots, x_N)$ is a function on $V^\vee$, the vector field $L_\theta$ in the first equation is the vector field on $V^\vee$ for the contragredient 
action of $G$ on $V^\vee$. This equation can be written as 
$$\frac{d}{dt}f(\exp(t\theta)\cdot \xi)=\beta(\theta) f(\exp(t\theta)\cdot \xi)$$ for any $\theta\in\mathfrak g$ and $\xi\in V^\vee$.  
Solving this ordinary differential equation, we get 
$$f(\exp(t\theta)\cdot\xi )=e^{\beta (\theta)t} f(\xi),$$ 
that is, $f$ is homogeneous of weight $\beta$ with respect to the group action of $G$ on $V^\vee$. Suppose $G$ is connected. Then
$e^{\theta}$ ($\theta\in \mathfrak g$) generate a dense subgroup of $G$. Suppose furthermore that $f$ is continuous. Then for any $\xi\in V^\vee$, 
the restriction of $f$ to the closure of the orbit $G\xi$ is uniquely 
determined by the value $f(\xi)$.  

Let $\rho: \mathfrak g\to 
\mathfrak {gl}(V)$ be the Lie algebra homomorphism defined by the representation $V$ of $G$, and write 
$$\rho(\theta)(x_i)=\sum_{j=1}^N \theta_{ij}x_j.$$
Then the equation $L_\theta f=\beta(\theta) f$ can be written as 
$$\sum_{i, j}\theta_{ji}x_i\frac{\partial f}{\partial x_j} +\beta(\theta)f=0.$$
Define formally the Fourier transform $\hat f$ of $f$ by 
$$\hat f(\xi_1, \ldots, \xi_N)=\int f(x_1, \ldots, x_N)e^{-(x_1\xi_1+\cdots+x_N\xi_N)}dx_1\cdots dx_N.$$ Formally we have 
$$\widehat {(x_i f)}=-\frac{\partial}{\partial\xi_i}\hat f, \quad \widehat{\Big(\frac{\partial}{\partial x_i} f\Big)}=\xi_i \hat f.$$
Taking the Fourier transform of the tautological system, we get
\begin{eqnarray*}
&& \sum_{i, j} \theta_{ji}\xi_j\frac{\partial \hat f}{\partial \xi_i} +(\mathrm{Tr}(\rho(\theta))-\beta(\theta))\hat f=0 \hbox{ for all } \theta\in \mathfrak g,\\
&&g(\xi_1,\ldots, \xi_N) \hat f=0 \hbox { for all } g\in I(\hat X).
\end{eqnarray*}
The first system of equations is just $$L_\theta \hat f= (\beta(\theta)-\mathrm{Tr}(\rho(\theta)))\hat f \hbox{ for all } \theta\in \mathfrak g,$$ and
it means that $\hat f$ is homogeneous of weight $\beta-\mathrm{Tr}(\rho)$ with respect to the action of $G$ on $V$. The 
second system of equations means that $\hat f$ is supported in $\hat X$. For a rigorous treatment via the Fourier transform, see \cite[\S 3]{FGKZ}.
By the Fourier inversion formula, to get a solution for the tautological system, we may start with a function on $V$ supported on $\hat X$ and homogenous 
of weight $\beta-\mathrm{Tr}(\rho)$ with respect to the action of $G$, and take its inverse Fourier transform. This is the method that we are going to use. 
In this paper, we work out a theory of tautological systems when the ground field is of characteristic $p$, a problem posed in 
\cite[\S 10]{LY}. In characteristic $p$, differential equations or $D$-modules do not behave well. Instead, we work with $\ell$-adic local systems, and more generally
objects in the derived category of $\overline{\mathbb Q}_\ell$-sheaves. The idea is that the tautological system is the Deligne-Fourier 
transform of a homogeneous object.

\medskip
\begin{remark} Let $G_0$ be an algebraic group acting on a smooth projective variety $X$ of dimension $d$,  
$\mathcal L$ a very ample $G_0$-linearized invertible sheaf, and
$V=\Gamma(X,\mathcal L)^\vee$. Then $G_0$ acts on $V$. Let $G=G_0\times\mathbb G_m$, let
$\mathbb G_m$ act trivially on $X$, and act on $V$ by scalar multiplication. We have a $G$-equivariant embedding $X\to \mathbb P(V)$.  
Let $\hat X$ be the cone over $X$, and let $\beta:\mathfrak g_0\oplus \mathbb C\to \mathbb C$ be a character, where $\mathfrak g_0$ is 
the Lie algebra of $G_0$. The tautological system $\tau(G,\beta, V, \hat X)$
\begin{eqnarray*}
&& L_\theta f =\beta(\theta) f \quad (\theta\in\mathfrak g_0\oplus \mathbb C),\\
&& g(\frac{\partial}{\partial x_1},\ldots, \frac{\partial}{\partial x_1}) f=0 \quad (g \hbox{ lies in the ideal sheaf of } \hat X),
\end{eqnarray*}
is studied in detail \cite{HLZ, LY}. Under certain conditions, period integrals for the family of hyperplane sections of $X$ are solutions of this system. 
More precisely, let $B$ be the Zariski 
open subset of $V^\vee=\Gamma(X,\mathcal L)$ such that for any $\phi\in B$, the hypersurface 
$Y_\phi=\{\phi=0\}$ of $X$ is smooth, let $$\mathcal Y=\{(x,\phi)\in X\times B|\phi(x)=0\},$$  let $\pi:\mathcal Y\to B$ be the projection, 
let $\omega_X$ and $\omega_{Y_\phi}$ be the invertible sheaves of top degree holomorphic forms on 
$X$ and $Y_\phi$, respectively, and let $i:Y_\phi\to X$ be the close immersion. By the adjunction formula, we have 
$$i^\ast (\mathcal L+\omega_X)\cong \omega_{Y_\phi}.$$
We thus have a restriction map 
$$\mathbf R_\phi: H^0(X, \mathcal L+ \omega_X) \to H^0(Y_\phi, \omega_{Y_\phi}).$$ 
Let $\tau\in H^0(X,
\mathcal L+\omega_X)$ be an eigenvector for the action of $\mathfrak g_0$, that is, 
there exists a character $\beta_0:\mathfrak g_0\to \mathbb C$ with the property $$\theta\cdot \tau
=\beta_0(\theta)\tau$$ for any $\theta\in\mathfrak g_0$. Let $\beta:\mathfrak g_0\oplus \mathbb C\to \mathbb C$ be the character 
defined by $\beta(\theta,\lambda)=\beta_0(\theta)+\lambda$, and let 
$\gamma_\phi\in H_{d-1}(Y_\phi,\mathbb Z)$ be locally constant homology $(d-1)$-cycles. Then by \cite[Theorem 8.8]{LY}, the period integral 
$$\int_{\gamma_\phi} \mathbf R_\phi(\tau)$$ as a function of $\phi$ is a solution of the tautological system $\tau(G,\beta, V, \hat X)$, 
provided that there exists
a principal $G$-equivariant $H$-bundle $M\to X$ with a CY-structure and a character $\chi: H\to \mathbb C^\ast$ such that 
$\mathcal L$ is the invertible sheaf corresponding to 
the line bundle $M\times_H \mathbb C\to X$, where $H$ acts on the factor $\mathbb C$ through the character $\chi$.
Here a CY-structure is a nowhere vanishing holomorphic form $\omega_M$ on $M$ 
of top degree which is an eigenvector of the action of $H$, that is $h\cdot \omega_M=\chi_M(h)\omega_M$ for some 
character $\chi_M:H\to \mathbb C^\ast$. 
\end{remark}

We first construct the homogenous objects. 
Let $k$ be a field, and let $\ell$ be a prime number distinct from the characteristic of $k$. For any
$k$-scheme $S$ of finite type, let $D_c^b(S,\overline{\mathbb Q}_\ell)$ be the derived category of $\overline{\mathbb Q}_\ell$-sheaves
on $S$.
Let $G'$ be a commutative algebraic group defined over $k$, and let 
$$m:G'\times_k G'\to G', \quad p_1:G'\times_k G'\to G', \quad p_2:G'\times_k G'\to G'$$ be the multiplication on $G'$ and the projections. 
A \emph{multiplicative sheaf} on $G'$ is a rank $1$ lisse $\overline{\mathbb Q}_\ell$-sheaf $\mathcal L$ together with an 
isomorphism $$\theta: p_1^* \mathcal L\otimes p_2^*\mathcal L\stackrel\cong\to m^*\mathcal L$$ such that the following conditions hold:

(1) Symmetry: Let $\sigma: G'\times_k G'\to G'\times_k G'$ and $\sigma': p_1^*\mathcal L\otimes p_2^*\mathcal L
\to p_2^*\mathcal L\otimes p_1^*\mathcal L$ be the morphisms defined by permuting factors. The following diagram commutes:
$$\begin{array}{ccccc}
 p_1^*\mathcal L\otimes p_2^*\mathcal L&&\stackrel\theta\to&& m^*\mathcal L\\
{\scriptstyle \sigma'} \downarrow&&&&\uparrow{\scriptstyle \cong} \\
p_2^*\mathcal L\otimes p_1^*\mathcal L&\cong& \sigma^*(p_1^*\mathcal L\otimes p_2^*\mathcal L)&
\stackrel{\sigma^*(\theta)}\to& \sigma^* m^* \mathcal L
\end{array}$$

(2) Associativity: Let $q_i:G'\times_k G'\times_k G'\to G'$ (resp. $q_{ij}:G'\times_k G'\times_k G'\to G'\times_kG'$) be the projections 
to the $i$-th factor (resp. $(i,j)$-th factor) for any $1\leq i\leq 3$ (resp. $1\leq i<j\leq 3$), and let 
$m_3: G'\times_k G'\times_k G'\to G'$ be the multiplication. The following diagram commutes:
$$\begin{array}{ccc}
q_1^*\mathcal L\otimes q_2^*\mathcal L\otimes q_3^*\mathcal L
&\stackrel{q_{12}^*(\theta)\otimes\mathrm{id}}\to& q_{12}^*m^*\mathcal L\otimes q_3^*\mathcal L\\
\scriptstyle{\mathrm{id} \otimes q_{23}^*(\theta)}\downarrow&&\downarrow \scriptstyle {(m q_{12}\times q_3)^* (\theta)}\\
q_1^*\mathcal L\otimes q_{23}^*m^*\mathcal L &\stackrel{(q_1\times mq_{23})^*(\theta)}\to&m_3^* \mathcal L.
\end{array}$$
The isomorphism $\theta$ induces an isomorphism 
$$p_1^* \mathcal L\otimes p_2^*\mathcal L\otimes m^*\mathcal L^{-1}\cong\overline{\mathbb Q}_\ell.$$
Restricting this isomorphism to $(e, e)$, where $e:\mathrm{Spec}\,k\to G'$ is the unit section of $G'$, we get an isomorphism
$$e^* \mathcal L\cong \overline{\mathbb Q}_\ell.$$ 
We can obtain multiplicative $\overline{\mathbb Q}_\ell$-sheaf on $G'$ is as follows: Let 
$$0\to N_i \to G'_i\to G'\to 0\quad (i\in I)$$ be the inverse system of all extensions of $G'$ by finite discrete abelian groups 
$N_i$. Note that each $G'_i\to G'$ is a Galois \'etale isogeny with Galois group $N_i$. We may regard the projective system
$\{G'_i\}_{i\in I}$ as the universal covering space classifying \'etale isogenies of $G'$. Set 
$$\pi_1^{\mathrm{isogeny}}(G')=\varprojlim_{i\in I} N_i.$$ 
It is a quotient of the \'etale fundamental group $\pi_1(G')$ of $G'$, and coincides with $\pi_1(G')$ if $k$ is algebraically closed of characteristic $0$. 
Let $\beta: \pi_1^{\mathrm{isogeny}}(G')\to \overline{\mathbb Q}_\ell^*$ be a character. It induces a character on the \'etale fundamental 
group of $G'$. Denote by $\mathcal L_\beta$ the corresponding $\overline{\mathbb Q}_\ell$-sheaf associated to this Galois representation. 
Then $\mathcal L_\beta$ is multiplicative. Any
multiplicative sheaf on $G'$ is of this form. See \cite[2.3.10]{MB} and \cite[2.9]{G} for a proof. 

Suppose $k=\mathbb F_q$ is a finite field with $q$ elements of characteristic $p$. Let $F$ be the Frobenius morphism, that is,  $F$ is identity on the 
underlying topological space of $G'$, and $F$ maps any section of $\mathcal O_G$ to its $q$-th power. Let 
$$L: G'\to G',\quad x\mapsto F(x)\cdot x^{-1}$$ be the Lang isogeny. 
Its kernel is the group $G'(\mathbb F_q)$ of $\mathbb F_q$-points of $G'$, and 
it is a Galois \'etale covering with Galois group $G'(\mathbb F_q)$. By \cite[VI \S 1 Proposition 6]{S}, any Galois isogeny over $G'$ is a quotient of the Lang isogeny. It follows that 
$$\pi_1^{\mathrm{isogeny}}(G')\cong G'(\mathbb F_q).$$ 
Let $\beta: G'(\mathbb F_q)\to \overline{\mathbb Q}_\ell^\ast$ be a character. The multiplicative sheaf 
$\mathcal L_\beta$ is called the \emph{Lang sheaf} associated to $\beta$. 
For any $x\in G'(\mathbb F_q)$, we have
$$(\mathrm{Frob}_x, \mathcal L_{\beta,\bar x})=\beta(x'),$$
where $\mathrm{Frob}_x$ is the geometric Frobenius element in $\pi_1(G')$.  

\begin{remark} Let $G'$ be a connected commutative complex algebraic group, let $\mathfrak g'$ be its Lie algebra, and let 
$\beta:\mathfrak g'\to\mathbb C$ be a character. The exponential map 
$$\mathfrak g'\to G', \quad \theta\mapsto \exp(\theta)$$ is an epimorphism. It is the universal covering space of $G'$, and its 
kernel is isomorphic to the fundamental group $\pi_1(G')$ of $G'$. Restricting $e^\beta$ to this kernel, we get a character
$e^\beta: \pi_1(G')\to \mathbb C^*$. The local system $\mathcal L_{e^\beta}$ associated to this character is a multiplicative sheaf. 
For any holomorphic function $f$ on $G'$ and any invariant vector field $\theta\in\mathfrak g$ on $G'$, let
$$\nabla_{\theta}(f)=\theta(f)-\beta(\theta)f.$$ Then $\nabla$ is an integrable connection on the trivial vector bundle over $G'$. 
Horizontal sections are those holomorphic functions $f$ such that 
$$\theta(f)=\beta(\theta)f,$$
that is, $f$ are homogenous of weight $\beta$. The multiplicative local system $\mathcal L_{e^\beta}$ is isomorphic to the local system of horizontal 
sections of this connection. 
\end{remark}

From now on, we assume $k$ is a perfect field of characteristic $p$ unless we state otherwise. 
Let $G$ be an algebraic group defined over $k$, and let $G\to \mathrm{GL}(V)$ be a representation, 
where $V$ is a finite dimensional vector space over $k$. 
By abuse of notation, we also denote 
by $V$ the affine $k$-scheme $\mathrm{Spec}(\mathrm{Sym}^\cdot (V^\vee))$. Let $v$ be a nonzero vector in 
$V$, and let $Q$ be the connected component of the stabilizer of $v$. 
Consider the morphism 
$$\iota: G/Q\to V, \quad gQ\mapsto gv.$$ 
It is quasi-finite. Let $G'=G/Q[G,G]$, which is a commutative algebraic group. Fix a multiplicative sheaf $\mathcal L_\beta$ 
on $G'$. Denote its inverse image by the canonical morphism $G/Q\to G'$ also by $\mathcal L_\beta$.
Note that $\iota_!\mathcal L_\beta$ and $R\iota_\ast\mathcal L_\beta$ are homogeneous objects in $D_c^b(V,
\overline{\mathbb Q}_\ell)$ supported in the $G$-invariant subset $\overline{Gv}$. By homogeneity, we mean that 
 $$\mu^\ast \iota_!\mathcal L_\beta\cong \mathcal L'_{\beta}\boxtimes \iota_!\mathcal L_\beta,\quad 
 \mu^\ast R\iota_\ast\mathcal L_\beta\cong \mathcal L'_{\beta}\boxtimes R\iota_\ast\mathcal L_\beta, $$
 where $\mu:G\times V\to V$ is the action of $G$ on $V$, and $\mathcal L'_{\beta}$ is the inverse image of $\mathcal L_\beta$ by
 the morphism $G\to G/Q$. (Confer the proof of Lemma \ref{equivariant}). 
Fixed a nontrivial additive character $\psi:\mathbb F_p\to\overline{\mathbb Q}_\ell^\ast$, and let 
$\mathscr F_\psi:D_c^b(V,\overline{\mathbb Q}_\ell)\to D_c^b(V^\vee, \overline{\mathbb Q}_\ell)$ be the Deligne-Fourier transform
defined by $\psi$. We define the \emph{$\ell$-adic tautological sheaves} to be 
\begin{eqnarray*}
\mathcal T_!(G, \beta, V, v,\psi)&=&\mathscr F_\psi(\iota_! \mathcal L_\beta[\mathrm{dim}\, G/Q]),\\
\mathcal T_{\ast}(G, \beta, V, v,\psi)&=&\mathscr F_\psi(R\iota_\ast \mathcal L_\beta[\mathrm{dim}\, G/Q]).
\end{eqnarray*}
Recall that for any vector bundle $E$ of rank $r$ over a $k$-scheme $S$ of finite type, the Deligne-Fourier transform is the functor 
$$\mathscr F_\psi:
D_c^b(E,\overline{\mathbb Q}_\ell)\to D_c^b(E^\vee,\overline{\mathbb Q}_\ell),\quad K\mapsto R\mathrm{pr}^\vee_!(\mathrm{pr}^\ast K\otimes 
\langle\,,\,\rangle^\ast \mathcal L_\psi)[r],$$ where $E^\vee\to S$ is the dual vector bundle of $E$,  $\mathrm{pr}:E\times_k E^\vee\to E$ 
and $\mathrm{pr}^\vee:E\times_kE^\vee\to E^\vee$ are the projections, and $\langle\,,\,\rangle:E\times_kE^\vee \to\mathbb A_S^1$ is the pairing, 
and $\mathcal L_\psi$ is the Artin-Schreier sheaf on $\mathbb A^1$ associated to the nontrivial additive character 
$\psi:\mathbb F_p\to\overline{\mathbb Q}_\ell^\ast$, that is, the pulling back to $\mathbb A_k^1$ of the Lang sheaf on 
the algebraic group $\mathbb A_{\mathbb F_p}^1$ associated to the character $\psi:\mathbb A^1(\mathbb F_p)\to\overline
{\mathbb Q}_\ell^\ast$. 
See \cite{L} for properties of the Fourier transform. For any $k$-scheme $a:S\to \mathrm{Spec}\,k$ of finite type,  
denote by $D_S:D_c^b(S,\overline{\mathbb Q}_\ell)\to D_c^b(S,\overline{\mathbb Q}_\ell)$ the Verdier dual functor $D_S=R\mathcal Hom(\cdot,Ra^! \overline{\mathbb Q}_\ell)$. 

\begin{proposition}  Let $n=\mathrm{dim}\, G/Q$ and $N=\mathrm{dim}\,V$.

(i) We have $D_{V^\vee}(\mathcal T_\ast(G, \beta, V, v,\psi))\cong \mathcal T_!(G, \beta^{-1}, V, v,\psi^{-1})(n+N)$ 

(ii) In the case where $k$ is a finite field, 
$\mathcal T_!(G, \beta, V, v,\psi)$ (resp. $D_{V^\vee}(\mathcal T_\ast(G, \beta, V, v,\psi))$ ) is mixed of weights $\leq   
n+N$ (resp. $\leq   -(n+N)$).
\end{proposition}

\begin{proof} 

(i) By \cite[1.3.2.2]{L}, we have 
\begin{eqnarray*}
D_{V^\vee}(\mathcal T_\ast(G, \beta, V, v,\psi))&=& D_{V^\vee}(\mathscr F_\psi(R\iota_\ast \mathcal L_\beta[n]))\\
&\cong& \mathscr F_{\psi^{-1}}(D_V(R\iota_\ast\mathcal L_\beta[n]))(N)\\
&\cong& \mathscr F_{\psi^{-1}} (\iota_! D_{G/Q}(\mathcal L_\beta[n]))(N)\\
&\cong& \mathscr F_{\psi^{-1}}(\iota_!\mathcal L_{\beta^{-1}}[n])(n+N)\\
&\cong&  \mathcal T_!(G, \beta^{-1}, V, v,\psi^{-1})(n+N).
\end{eqnarray*}

(ii)  $\iota_! \mathcal L_\beta[n]$ is mixed of weights $\leq n$. Its Deligne-Fourier transform 
$\mathcal T_!(G, \beta, V, v)$ is mixed of weights $\leq n+N$ by
 \cite[Th\'eor\`eme 3.3.1]{D}. 
Then by (i), $D_{V^\vee}(\mathcal T_\ast(G, \beta, V, v,\psi))$ is mixed of weights $\leq -(n+N)$. 
\end{proof}

\begin{example}\label{GKZ} Let $(w_{ij})$ be an $n\times N$ matrix of rank $n$ with integer entries.
Take $G$ to be the torus $G=\mathbb G_{m,k}^n$. Consider the representation $V$ of $G$ defined by the action
\begin{eqnarray*}
\mathbb G_{m,k}^n\times_k  \mathbb A_{k}^N &\to& \mathbb
A_{k}^N, \\
\Big((t_1,\ldots, t_n),(x_1,\ldots, x_N)\Big)&\mapsto&
(t_1^{w_{11}}\cdots t_n^{w_{n1}}x_1,\ldots, t_1^{w_{1N}}\cdots
t_n^{w_{nN}}x_N).
\end{eqnarray*}
The connected component of the stabilizer $v=(1,\ldots,1)\in V$ is trivial. 
Let $\hat X$ the the closure of the orbit of $v$. Then $\hat X$ is 
the closed subscheme of $\mathbb A_k^N$ defined by $\prod_{a_j>0}
{\xi_j}^{a_j}-\prod_{a_j<0}{\xi_j}^{-a_j}=0$, where $(a_1,\ldots,
a_N)\in\mathbb Z^N$ goes over the family of integral linear
relations
$$\sum_{j=1}^N
a_j{\mathbf w}_j=0$$ among ${\mathbf w}_1,\ldots, {\mathbf w}_N$ with ${\mathbf w}_j=(w_{1j},\ldots, w_{nj})\in\mathbb Z^n$. 

Suppose $k$ is the complex field $\mathbb C$. Let $\beta$ the character 
$$\beta:\mathbb C^n\to \mathbb C,\quad (\lambda_1,\ldots, \lambda_n)\mapsto \gamma_1\lambda_1+\cdots +\gamma_n\lambda_n,$$
where $\gamma_1,\ldots, \gamma_n\in \mathbb C$. 
The tautological system $\tau(G,\beta,V, \hat X)$ is exactly the 
Gelfand-Kapranov-Zelevinsky (GKZ) hypergeometric system 
\begin{eqnarray*}
&&\sum_{j=1}^N w_{ij} x_j\frac{\partial f}{\partial x_j}+\gamma_i
f=0 \quad (i=1,\ldots, n),\\ && \prod_{a_j>0}
\left(\frac{\partial}{\partial x_j}\right)^{a_j}f= \prod_{a_j<0}
\left(\frac{\partial}{\partial x_j}\right)^{-a_j}f
\end{eqnarray*} introduced in \cite{GKZ}. It is the system of differential equations satisfied by the exponential integral (period integral)
$$f(x_1, \ldots, x_n)=\int_{\sigma}t_1^{\gamma_1}\cdots t_n^{\gamma_n}e^{\sum_{j=1}^N x_j t_1^{w_{1j}}\cdots t_n^{w_{nj}}}\frac{dt_1}{t_1}\cdots
\frac{dt_n}{t_n},$$ where $\sigma$ is any real $n$-dimensional cycles in $(\mathbb C^*)^n.$

Suppose $k=\mathbb F_q$ is a finite field with $q$ elements. 
Let $\chi_1,\ldots, \chi_n:\mathbb F_q^\ast\to \overline{\mathbb Q}_\ell^\ast$ be multiplicative characters. 
They define a character $$\beta:{\mathbb G}_m^n(\mathbb F_q)\to \overline{\mathbb Q}_\ell^\ast,\quad (t_1,\ldots,t_n)\mapsto \chi_1(t_1)\cdots \chi_n(t_n).$$ 
The the tautological system $\mathcal T_!(G, \beta, V, v,\psi)$ is the GKZ hypergeometric sheaf studied in \cite{F}. 
For any $\mathbb F_q$-rational point $x=(x_1, \ldots, x_N)$ of $V^\vee$, we have 
$$\mathrm{Tr}(\mathrm{Frob}_x, \mathcal T_!(G, \beta, V, v,\psi)_{\bar x})=\sum_{t_1, \ldots, t_n\in \mathbb F_q^*} \chi_1(t_1)\cdots\chi_n(t_n)
\psi\Big(\mathrm{Tr}_{\mathbb F_q/\mathbb F_p}\Big(\sum_{j=1}^N x_j t_1^{w_{1j}}\cdots t_n^{w_{nj}}\Big)\Big).$$ 
The exponential sum on the righthand side is the arithmetic counterpart of the above exponential integral.
\end{example}

\medskip
We will give more examples of tautological systems in \S 2. In the following, we focus on the tautological systems constructed from
a quasi-projective homogenous variety embedded equivariantly in $\mathbb P(V)$.
More precisely, let $G_0$ be an algebraic group, let $G_0\to\mathrm{GL}(V)$ be a representation, and
let $G=G_0\times \mathbb G_m$. Let $G$ act on $V$ so that the factor $\mathbb G_m$ acts on $V$ by scalar multiplication. 
For any nonzero vector $x\in V$, denote by $[x]$ the corresponding point in 
the projective space $\mathbb P(V)=\mathrm{Proj}(\mathrm{Sym}(V^\vee))$, and denote by $\mathbb L_{[x]}$ the line
in $V$ spanned by $x$. Let $\mathbb L$ be the tautological line bundle over $\mathbb P(V)=\mathrm{Proj}(\mathrm{Sym}(V^\vee))$. 
The fiber of $\mathbb L$ at $[x]$ can be identified with the line $\mathbb L_{[x]}$. Note that $G$ acts on 
$\mathbb P(V)$ and on $\mathbb L$, and the factor 
$\mathbb G_m$ acts trivially on $\mathbb P(V)$ and acts by scalar multiplication on fibers of $\mathbb L$. Let $\mathbb L^\vee$ be the dual 
line bundle of $\mathbb L$. It is provided with the contragrediant action by $G$. The fiber of $\mathbb L^\vee$ over $[x]$ consists of 
linear functionals on the line 
$\mathbb L_{[x]}$.  The canonical morphisms $\mathbb L\to \mathbb P(V)$ and 
$\mathbb L^\vee\to \mathbb P(V)$ are $G$-equivariant. 

We have a canonical embedding  $$i:\mathbb L \to V\times_k \mathbb P(V)$$ of $\mathbb L$ into the trivial vector bundle 
$V\times_k\mathbb P(V)\to\mathbb P(V)$. 
The fiber of this morphism over 
$[x]\in\mathbb P(V)$ is the inclusion of the line $\mathbb L_{[x]}$ into $V$.  The transpose of $i$ is the evaluation morphism $$\mathrm{ev}:
V^\vee\times_k\mathbb P(V)\to \mathbb L^\vee.$$ For any $\phi\in V^\vee$, the fiber of $\mathrm{ev}$ over $[x]$ maps  
a point $(\phi, [x])$ in $V^\vee\times_k\mathbb P(V)$ to the restriction $\phi|_{\mathbb L_{[x]}}\in \mathbb L^\vee_{[x]}$ of $\phi$ to the subspace $\mathbb L_{[x]}$. 
Let ${\mathbb L}^\circ$ be the complement of the zero section of $\mathbb L$. 
We have an isomorphism $$V-\{0\}\stackrel\cong\to {\mathbb L}^\circ$$ which maps every nonzero vector $x$ to the point $x$ 
considered as an element in the line $\mathbb L_{[x]}$.   ${\mathbb L}^\circ$ is invariant under the action of $G$. 

Let $v\in V$ be a nonzero vector vector, let $Q\subset G$ be the connected component of the stabilizer of $v$, and let $P=Q\mathbb G_m$. Note that 
$P$ stabilizes $[v]\in \mathbb P(V)$. 
Regard $v$ as an element in $\mathbb L_{[v]}$. 
Let $v^\ast$ be the element in $(\mathbb L^\vee)_{[v]}$ define by $v^\ast(v)=1$. Note that $Q$ is also the connected component of the 
stabilizer of $v^\ast$. Let 
$$\mathbb L_{G/P}=\mathbb L\times_{\mathbb P^n(V)}G/P,\quad \mathbb L^\vee_{G/P}=\mathbb L^\vee\times_{\mathbb P^n(V)} G/P,\quad  
i_{G/P}:\mathbb L_{G/P}\to V\times_k G/P,\quad \mathrm{ev}_{G/P}: V^\vee\times_k G/P\to \mathbb L^\vee_{G/P}$$ the base changes of 
$\mathbb L$, $\mathbb L^\vee$, $i$ and $\mathrm{ev}$ respectively 
with respect to the morphism $$G/P\to \mathbb P(V),\quad gP\mapsto g[v].$$ 
Let $\kappa:G/Q\to \mathbb L_{G/P}^\circ$ (resp. $\kappa^\vee:G/Q\to \mathbb L_{G/P}^{\vee\circ}$) be the morphism $gQ\mapsto gv$ 
(resp. $gQ\mapsto gv^\ast$), where we regard $v$ (resp. $v^\ast$) as an element in  
the fiber $\mathbb L_{G/P,eP}\cong\mathbb L_{[v]}$ (resp.  ${\mathbb L}^\vee_{G/P,eP}\cong
{\mathbb L}^\vee_{[v]}$). 

\begin{lemma}\label{isomorphism} $\kappa: G/Q\to  \mathbb L^\circ _{G/P}$ and $\kappa^\vee: G/Q\to \mathbb L_{G/P}^{\vee\circ}$
are isomorphisms. 
\end{lemma}

\begin{proof} Note that $Q$ is exactly the stabilizer of $(v,eP)$ considered as a point in $\mathbb L^\circ _{G/P}$ with respect to the action of 
$G$ on $\mathbb L^\circ_{G/P}$. Moreover, the action of $G$ on $\mathbb L^\circ_{G/P}$ is transitive. We claim that the morphism 
$$G\to \mathbb L^\circ_{G/P}, \quad g\mapsto g\cdot (v, eP)$$ is smooth. It suffices to show it induces an epimorphism on tangent spaces
$$T_eG\to T_{(v, eP)}  (\mathbb L^\circ_{G/P}).$$ 
The tangent space $T_{(v, eP)}  (\mathbb L^\circ_{G/P})$ is isomorphic to the direct sum of $T_{eP}(G/P)$ and the one dimensional space corresponding 
to the fiber of the line bundle $\mathbb L_{G/P}$ over $eP$. The canonical homomorphism $T_eG\to T_e(G/P)$ is surjective. The tangent vectors along 
the direct factor $\mathbb G_m$ of $G$ are mapped surjectively onto the above one dimensional subspace. This proves our claim. By \cite[5.5.4 and 5.5.5]{S},
$\mathbb L^\circ _{G/P}$ is isomorphic to $G/Q$. Similarly,  $\mathbb L_{G/P}^{\vee\circ}$ is also isomorphic to $G/Q$. 
\end{proof} 

We use the isomorphism in the above lemma to identify $G/Q$ with  $\mathbb L^\circ _{G/P}$ (resp. $\mathbb L_{G/P}^{\vee\circ}$). By abuse of notation, we denote 
the sheaf $\kappa_!\mathcal L_\beta$ (resp. $\kappa^\vee_!\mathcal L_\beta$) on $\mathbb L^\circ _{G/P}$ (resp. $\mathbb L_{G/P}^{\vee\circ}$) also by 
$\mathcal L_\beta$. 
Let $j_{\mathbb L_{G/P}^\circ}:\mathbb L_{G/P}^\circ\hookrightarrow 
\mathbb L_{G/P}$ 
and $j_{\mathbb L_{G/P}^{\vee\circ}}:
\mathbb L_{G/P}^{\vee\circ}\hookrightarrow \mathbb L_{G/P}^\vee$ be the open immersions of the complements of zero sections, and
let $\mathrm{pr}:V^\vee\times_k {G/P}\to V^\vee$ be the projection. 
Let $\mathcal L_\beta|_{\mathbb G_m}$ (resp. $\mathcal L_\psi|_{\mathbb G_m}$) be the inverse image of 
$\mathcal L_\beta$ (resp. $\mathcal L_\psi$) under the canonical morphism $\mathbb G_m\to G'=G/Q[G,G]$ (resp. $\mathbb G_m\to \mathbb A^1$).
Let $G_!(\beta,\psi)$ and $G_*(\beta, \psi)$ be the sheaves on $\mathrm{Spec}\,k$ 
defined by the following Galois representations:
\begin{eqnarray*}
G_!(\beta,\psi)&=&\left\{\begin{array}{cl}
H^1_c(\mathbb G_{m,\bar k}, 
\mathcal L_\beta|_{\mathbb G_m}\otimes \mathcal L_\psi|_{\mathbb G_m})& \hbox{if } \mathcal L_\beta|_{\mathbb G_{m,\bar k}} \hbox { is nontrivial},\\
\overline{\mathbb Q}_\ell & \hbox{otherwise},
\end{array}\right. \\
G_*(\beta,\psi)&=&\left\{\begin{array}{cl}
H^1(\mathbb G_{m,\bar k}, 
\mathcal L_\beta|_{\mathbb G_m}\otimes \mathcal L_\psi|_{\mathbb G_m})& \hbox{if } \mathcal L_\beta|_{\mathbb G_{m,\bar k}} \hbox { is nontrivial},\\
\overline{\mathbb Q}_\ell(-1) & \hbox{otherwise},
\end{array}\right.
\end{eqnarray*}
By Poincar\'e duality, we have 
$$G_*(\beta,\psi)\cong G_!(\beta,\psi)^\vee (-1).$$ Actually if $\mathcal L_\beta|_{\mathbb G_{m,\bar k}}$ is nontrivial, then
$$H^1_c(\mathbb G_{m,\bar k}, 
\mathcal L_\beta|_{\mathbb G_m}\otimes \mathcal L_\psi|_{\mathbb G_m})\cong H^1(\mathbb G_{m,\bar k}, 
\mathcal L_\beta|_{\mathbb G_m}\otimes \mathcal L_\psi|_{\mathbb G_m})$$ and hence $G_!(\beta,\psi)=G_*(\beta, \psi)$ in this case.  
We will show that $G_\ast(\beta,\psi)$ and $G_!(\beta, \psi)$ are of rank $1$. (Confer Lemma \ref{fiber}). 
Denote the inverse images of $G_!(\beta,\psi)$ and $G_*(\beta,\psi)$ on any $k$-scheme by the same notation. 
In the case where $k=\mathbb F_q$ is a finite field with $q$ elements and  $\mathcal L_\beta|_{\mathbb G_{m,\bar k}}$ is nontrivial, by the 
Grothendieck trace formula, $G_\ast(\beta,\psi)$ and $G_\ast(\beta,\psi)$ is defined by the Galois representation
$$\mathrm{Gal}(\overline {\mathbb F}_q/\mathbb F_q)\to \overline{\mathbb Q}_\ell^\ast,\quad \mathrm{Frob}_q\mapsto 
-\sum_{t\in \mathbb F_q^\ast}\beta(t)\psi(t),$$ where $\mathrm{Frob}_q$ is the geometric Frobenius element in 
$\mathrm{Gal}(\overline {\mathbb F}_q/\mathbb F_q)$, and the 
righthand side is nothing but the Gauss sum. 

Let $U=\mathrm{ev}_{G/P}^{-1}(\mathbb L_{G/P}^{\vee\circ})$ and let $H$ be the complement of $U$ in
$V^\vee\times_k G/P$. Then $U$ (resp. $H$) is an open (resp. closed) 
subset 
of $V^\vee\times {G/P}$, and a rational point $(\phi,gP)$ of $V^\vee\times G/P$ lies in $U$ (resp. $H$) if and only if $\phi(g(v))\not=0$ (resp. $\phi(g(v))=0$). 
For any rational point $\phi$ of 
$V^\vee$, the fiber $H_\phi$ of $H$ over $\phi$ is the hyperplane section $\{gP\in G/P | \phi(gv)=0\}$ in $G/P$, and 
the fiber $U_\phi$ of $U$ over $\phi$ is the complement in $G/P$ of the hyperplane section. 
Fix notations by the following commutative diagram:
$$\begin{array}{rccrl}
G/Q\cong \mathbb L_{G/P}^{\vee\circ}&\stackrel {\mathrm{ev}'_{G/P}}\leftarrow &U&&H\\
{\scriptstyle j_{\mathbb L_{G/P}^{\vee\circ}}}\downarrow&&\downarrow{\scriptstyle j'}&{\scriptstyle i'}\swarrow&\\
\mathbb L^\vee_{G/P}&\stackrel {\mathrm{ev}_{G/P}}\leftarrow &V^\vee \times_k (G/P)&\stackrel {\mathrm{pr}}\to & V^\vee\\
&\searrow&\downarrow&&\downarrow\\
&&G/P&\to&\mathrm{Spec}\,\mathbb F_q
\end{array}$$

\begin{theorem} \label{mainthm} Notation as above. Suppose $Q$ is geometrically connected. 

(i) We have 
\begin{eqnarray*}
\mathcal T_!(G, \beta, V, v,\psi)&\cong&R\mathrm{pr}_! \mathrm{ev}_{G/P}^\ast R  j_{\mathbb L_{G/P}^{\vee\circ},*} \mathcal L_\beta
[n+N-1]\otimes G_!(\beta,\psi),\\
\mathcal T_\ast (G, \beta, V, v,\psi)&\cong&R\mathrm{pr}_\ast \mathrm{ev}_{G/P}^\ast j_{\mathbb L_{G/P}^{\vee\circ},!} \mathcal L_\beta
[n+N-1]\otimes G_*(\beta,\psi)\\
&\cong&R\mathrm{pr}_\ast j'_!  \mathrm{ev}_{G/P}^{\prime\ast} \mathcal L_\beta
[n+N-1]\otimes G_*(\beta,\psi).
\end{eqnarray*}

(ii) Suppose furthermore that there exists a multiplicative sheaf $\mathcal L_{\beta_0}$ on $G/P[G,G]$ such that $\mathcal L_\beta$ 
is isomorphic to its inverse image under the canonical morphism 
$G/Q[G,G]\to G/P[G,G]$. Then we have 
\begin{eqnarray*}
\mathcal T_\ast (G, \beta, V, v,\psi)&\cong&R\mathrm{pr}_\ast j'_! (\mathcal L_{\beta_0}|_{U})(-1)
[n+N-1],
\end{eqnarray*}
and we have a distinguished triangle 
$$R\mathrm{pr}_! Ri'_* (\mathcal L_{\beta_0}|_{H})(-1)
[n+N-3]\to R\mathrm{pr}_!  (\mathcal L_{\beta_0}|_{V^\vee\times_k G/P})
[n+N-1]\to  \mathcal T_!(G, \beta, V, v,\psi)\to,$$
where $\mathcal L_{\beta_0}|_{U}$, $\mathcal L_{\beta_0}|_{H}$ and $(\mathcal L_{\beta_0}|_{V^\vee\times_k G/P})$ 
are the inverse images of the sheaf $\mathcal L_{\beta_0}$ on $G/P$.
\end{theorem}

Combined with the proper base change theorem, we have the following corollary which 
shows that tautological systems parameterize the cohomology of $U_\phi$. 

\begin{corollary}\label{maincor} Suppose $G/P$ is proper and $Q$ is geometrically connected. 

(i) For any point $\phi$ in $V^\vee$, we have
$$\mathcal T_\ast(G,\beta,V,v,\psi)_{\bar \phi}=R\Gamma_c(U_{\bar \phi}, (\mathrm{ev}_{G/P}^{\prime\ast} 
\mathcal L_\beta)|_{U_{\bar \phi}})[n+N-1]\otimes 
G_*(\beta,\psi).$$ 

(ii) Suppose that there exists a multiplicative sheaf $\mathcal L_{\beta_0}$ on $G/P[G,G]$ such that $\mathcal L_\beta$ 
is isomorphic to its inverse image under the canonical morphism 
$G/Q[G,G]\to G/P[G,G]$. Then any point $\phi$ in $V^\vee$, we have
$$\mathcal T_\ast(G,\beta,V,v,\psi)_{\bar \phi}=R\Gamma_c(U_{\bar \phi}, \mathcal L_{\beta_0}|_{U_{\bar \phi}})(-1)[n+N-1].$$ 
\end{corollary} 

\begin{proof} Follows from the formula for $\mathcal T_\ast(G,\beta,V,v,\psi)$ in Theorem \ref{mainthm} 
and the proper base change theorem. 
\end{proof}

\begin{remark} Suppose $k=\mathbb F_q$ is a finite field. For any $\mathbb F_q$-point $x$ of $G/P$, the fiber 
$\pi^{-1}(x)$ of the projection $\pi:G\to G/P$ is a principal $P$-homogenous space. By \cite[VI \S1 Corollary 1]{S}, $\pi^{-1}(x)$ has a $\mathbb F_q$-point. 
Thus any $\mathbb F_q$-point of $G/P$ is the image of a $\mathbb F_q$-point of $G$. For any $\mathbb F_q$-point $g$ of 
$G$, denote its image in $G/P$ by $gP$. Suppose $\mathcal L_\beta$ is the Lang sheaf defined by a character
$\beta: (G/Q[G,G])(\mathbb F_q)\to\overline{\mathbb Q}_\ell^*$. Then by the Grothendieck trace formula and Corollary \ref{maincor}, for 
any $\mathbb F_q$-point $\phi$ of $V^\vee$, we have 
\begin{eqnarray*}
\mathrm{Tr}(\mathrm{Frob}_\phi, \mathcal T_\ast(G,\beta,V,v,\psi)_{\bar \phi})
=(-1)^{n+N-1}G_*(\beta, \psi)\sum_{x\in (G/P)(\mathbb F_q),\; \phi(g_xv)\not=0}\beta(\phi(g_xv)^{-1}\cdot g_x),
\end{eqnarray*}
where for any $\mathbb F_q$-point $x$ of $G/P$,  $g_x$ is a $\mathbb F_q$-point of $G$ in the fiber $\pi^{-1}(x)$, 
$\phi(g_xv)^{-1}$ is considered as a $\mathbb F_q$-point of 
$\mathbb G_m\subset G$, $\beta(\phi(g_xv)^{-1}\cdot g_x)$ is the value of $\beta$ at the image of the $\mathbb F_q$-point
$\phi(g_xv)^{-1}\cdot g_x$ in $G/Q[G,G]$, and $G_*(\beta, \psi)$ is the Gauss sum. Note that the value $\beta(\phi(g_xv)^{-1}\cdot g_x)$ is independent of the choice of the 
$\mathbb F_q$-points $g_x$ in the fiber $\pi^{-1}(x)$. If 
there exists a character 
$\beta_0:(G/P[G,G])(\mathbb F_q)\to\overline{\mathbb Q}_\ell^\ast$ such that 
$\beta$ is the composite $$(G/Q[G,G])(\mathbb F_q)\to (G/P[G,G])(\mathbb F_q)\stackrel{\beta_0}\to
\overline{\mathbb Q}_\ell^\ast.$$  Then 
\begin{eqnarray*}
\mathrm{Tr}(\mathrm{Frob}_\phi, \mathcal T_\ast(G,\beta,V,v,\psi)_{\bar \phi})
&=&(-1)^{n+N-1}q\sum_{x\in (G/P)(\mathbb F_q),\; \phi(g_xv)\not=0}\beta_0(g_x),\\
\mathrm{Tr}(\mathrm{Frob}_\phi, \mathcal T_!(G,\beta,V,v,\psi)_{\bar \phi})
&=&(-1)^{n+N-1}\sum_{x\in (G/P)(\mathbb F_q)}\beta_0(g_x)-
(-1)^{n+N-1}q\sum_{x\in (G/P)(\mathbb F_q),\; \phi(g_xv)=0}\beta_0(g_x).
\end{eqnarray*}
The second equation follows from the distinguished triangle in Theorem \ref{mainthm} (ii). 
\end{remark}

\section{Proof of Theorem \ref{mainthm}} 

Let $\kappa_0:G/Q\to V-\{0\}$ be the morphism $gQ\mapsto gv$, 
let $j_0:V-\{0\}\hookrightarrow V$ be the open immersion, and let $\pi:V\times G/P \to V$ be the projection. 
The following diagram commutes:
$$\begin{array}{rcccl}
&&G/Q&&\\
&\stackrel{ \kappa}{\underset\cong\swarrow}&&\stackrel{\kappa_0}\searrow&\\
\mathbb L_{G/P}^\circ&& && V-\{0\}\\
{\scriptstyle j_{\mathbb L_{G/P}^\circ}}\downarrow&&&&\downarrow{\scriptstyle j_0}\\
\mathbb L_{G/P}&\stackrel {i_{G/P}}\to &V\times_k (G/P)&\stackrel \pi\to & V\\
&\searrow&\downarrow&&\downarrow\\
&&G/P&\to&\mathrm{Spec}\,\mathbb F_q
\end{array}$$
We have $\iota=j_0\kappa_0$, and hence 
\begin{eqnarray*}
\mathcal T_!(G, \beta, V, v,\psi)=\mathscr F_\psi(j_{0!}\kappa_{0!} \mathcal L_\beta[n])
\cong  \mathscr F_\psi(R\pi_! Ri_{G/P,!} j_{\mathbb L_{G/P}^\circ, !} \mathcal L_\beta[n]).
\end{eqnarray*}
By \cite[1.2.3.5 and 1.2.2.4]{L}, we have 
$$\mathscr F_\psi R\pi_!\cong R\mathrm{pr}_! \mathscr F_{\psi, V\times_k (G/P)},\quad 
\mathscr F_{\psi, V\times_k (G/P)}Ri_{G/P,!}\cong  \mathrm{ev}_{G/P}^\ast \mathscr F_{\psi,\mathbb L_{G/P}}[N-1],$$
where  
$\mathscr F_{\psi,V\times_k (G/P)}$ is the Deligne-Fourier 
transform for the trivial vector bundle $V\times_k (G/P)\to G/P$ and 
$\mathscr F_{\psi,\mathbb L_{G/P}} $ is the Deligne-Fourier transform for the vector bundle $\mathbb L_{G/P}\to G/P$.  We thus have 
$$\mathcal T_!(G, \beta, V, v,\psi)\cong  \mathscr F_\psi(R\pi_!R i_{G/P,!} j_{\mathbb L_{G/P}^\circ, !}\mathcal L_\beta[n])\cong 
R\mathrm{pr}_!\mathrm{ev}_{G/P}^\ast \mathscr F_{\psi,\mathbb L_{G/P}}
(j_{\mathbb L_{G/P}^\circ, !}\mathcal L_\beta)[n+N-1].$$
The assertion for $\mathcal T_!(G, \beta, V, v)$ in Theorem \ref{mainthm} (i) then follows from Lemma \ref{mainlemma} below. 

Using \cite[1.3.1.1]{L}, the smooth base change theorem, and the same
proof as  \cite[1.2.3.5]{L}, one can show 
$$\mathscr F_\psi R\pi_\ast = R\mathrm{pr}_\ast \mathscr F_{\psi, V\times_k (G/P)}.$$
Since $i_{G/P}$ is a closed immersion we have $Ri_{G/P,\ast}=i_{G/P,!}$, and hence by \cite[1.2.2.4]{L}, we have  
$$\mathscr F_{\psi, V\times_k (G/P)}Ri_{G/P,\ast} \cong  \mathrm{ev}_{G/P}^\ast \mathscr F_{\psi,\mathbb L_{G/P}}[N-1].$$ So we have 
\begin{eqnarray*}
\mathcal T_\ast(G, \beta, V, v,\psi)&\cong&
 \mathscr F_\psi(Rj_{0\ast}R\kappa_{0\ast} \mathcal L_\beta[n])\\
&\cong&  \mathscr F_\psi(R\pi_\ast Ri_{G/P,\ast} Rj_{\mathbb L_{G/P}^\circ, \ast}\mathcal L_\beta[n])\\
&\cong& R\mathrm{pr}_\ast \mathrm{ev}_{G/P}^\ast \mathscr F_{\psi,\mathbb L_{G/P}}
(Rj_{\mathbb L_{G/P}^\circ, \ast } \mathcal L_\beta)[n+N-1].
\end{eqnarray*}
The assertion for $\mathcal T_\ast(G, \beta, V, v,\psi)$ in Theorem \ref{mainthm} (i) also follows from Lemma \ref{mainlemma} below and the proper base 
change theorem . 

Suppose furthermore that there exists a multiplicative sheaf $\mathcal L_{\beta_0}$ on $G/P[G,G]$ such that $\mathcal L_\beta$ 
is isomorphic to its inverse image under the canonical morphism 
$G/Q[G,G]\to G/P[G,G]$. Then we have 
$$G_*(\beta,\psi)\cong \overline{\mathbb Q}_\ell,\quad G_!(\beta,\psi)\cong \overline{\mathbb Q}_\ell(-1).$$ 
Moreover, by (i) and the proper base change theorem, we have 
\begin{eqnarray*}
\mathcal T_\ast (G, \beta, V, v,\psi)&\cong&R\mathrm{pr}_\ast \mathrm{ev}_{G/P}^\ast j_{\mathbb L_{G/P}^{\vee\circ},!} (\mathcal L_{\beta_0}|_{\mathbb L_{G/P}^{\vee\circ}})
(-1)[n+N-1]\\
&\cong &R\mathrm{pr}_\ast j'_! \mathrm{ev}_{G/P}^{\prime\ast}(\mathcal L_{\beta_0}|_{\mathbb L_{G/P}^{\vee\circ}})
(-1)[n+N-1]\\
&\cong& R\mathrm{pr}_\ast j'_! (\mathcal L_{\beta_0}|_{U'})(-1)[n+N-1]..
\end{eqnarray*}
Let $i:G/P\to \mathbb L_{G/P}^{\vee}$ be the zero section. We have $$Ri^! (\mathcal L_{\beta_0}|_{\mathbb L_{G/P}^{\vee}})\cong \mathcal L_{\beta_0}(-1)[-2]$$
and hence we have a distinguished triangle 
$$i_* \mathcal L_{\beta_0}(-1)[-2] \to\mathcal L_{\beta_0} |_{\mathbb L_{G/P}^{\vee}}\to 
R  j_{\mathbb L_{G/P}^{\vee\circ},*} (\mathcal L_{\beta_0}|_{\mathbb L_{G/P}^{\vee\circ}})\to.$$
Applying $\mathrm{ev}^*_{G/P}$, we get a distinguished triangle
$$i'_*(\mathcal L_{\beta_0}|_{H})(-1)[-2]\to (\mathcal L_{\beta_0} )|_{V^\vee\times_k G/P}\to 
\mathrm{ev}_{G/P}^\ast Rj_{\mathbb L_{G/P}^{\vee\circ},*} (\mathcal L_{\beta_0}|_{\mathbb L_{G/P}^{\vee\circ}})
\to.$$ 
Applying $R\mathrm{pr}_!$, we get the distinguished triangle 
$$R\mathrm{pr}_! Ri'_* (\mathcal L_{\beta_0}|_{H})(-1)
[n+N-3]\to R\mathrm{pr}_!  (\mathcal L_{\beta_0}|_{V^\vee\times_k G/P})
[n+N-1]\to  \mathcal T_!(G, \beta, V, v,\psi)\to.$$

\medskip
In the rest of this section, we prove Lemma \ref{mainlemma} which gives formulas for $ \mathscr F_{\psi,\mathbb L_{G/P}}
(j_{\mathbb L_{G/P}^\circ, ! }\mathcal L_\beta)$ and  $\mathscr F_{\psi,\mathbb L_{G/P}}
(Rj_{\mathbb L_{G/P}^\circ, \ast } \mathcal L_\beta)$

\begin{lemma} \label{equivariant} Let $m:G\times_k \mathbb L_{G/P}\to \mathbb L_{G/P}$ be the morphism 
$(g,x)\mapsto gx$. Denote the inverse image of 
$\mathcal L_\beta$ by the morphism $G\to G/Q$ by $\mathcal L'_\beta$. We have 
\begin{eqnarray*}
m^\ast j_{\mathbb L_{G/P}^\circ,!}\mathcal L_\beta \cong \mathcal L'_\beta\boxtimes j_{\mathbb L_{G/P}^\circ,!}\mathcal L_\beta,\quad
m^\ast Rj_{\mathbb L_{G/P}^\circ,\ast}\mathcal L_\beta \cong 
\mathcal L'_\beta\boxtimes Rj_{\mathbb L_{G/P}^\circ,\ast}\mathcal L_\beta.
\end{eqnarray*}
\end{lemma} 

\begin{proof} We prove the second statement. The proof for the first statement is similar. 

Let $m':G\times_k G/Q\to G/Q$ be the morphism $(g,x)\mapsto gx$. We have 
$$m^{\prime\ast} \mathcal L_\beta\cong \mathcal L'_\beta\boxtimes\mathcal L_\beta.$$
Let $\phi:G\times_k \mathbb L_{G/P}\to G\times_k \mathbb L_{G/P}$ be the 
isomorphism defined by $(g,x)\mapsto (g, gx)$ and fix notation by the following 
commutative diagram:
$$\begin{array}{ccrcrcr}
G&\stackrel{\pi'_1}\leftarrow& G\times_k G/Q&\stackrel{\phi'}\to & G\times_k G/Q&\stackrel{\pi_2'}\to& G/Q\\
\parallel&&{\scriptstyle \mathrm{id}_G\times (j_{\mathbb L_{G/P}^\circ}\kappa)}\downarrow
&&{\scriptstyle \mathrm{id}_G\times  (j_{\mathbb L_{G/P}^\circ}\kappa)}\downarrow &&
{\scriptstyle  j_{\mathbb L_{G/P}^\circ}\kappa}\downarrow \\
G&\stackrel{\pi_1}\leftarrow& G\times_k  \mathbb L_{G/P}&\stackrel{\phi}\to & G\times_k  \mathbb L_{G/P}&\stackrel{\pi_2}\to& \mathbb L_{G/P}.
\end{array}$$
We have $m=\pi_2  \phi$. It follows that $m$ is smooth. 
By the smooth base change theorem and the projection formula, we have 
\begin{eqnarray*}
m^\ast R(j_{\mathbb L_{G/P}^\circ}\kappa)_\ast\mathcal L_\beta&\cong& 
R(\mathrm{id}_G\times (j_{\mathbb L_{G/P}^\circ}\kappa))_\ast m'^\ast \mathcal L_\beta\\
&\cong& R(\mathrm{id}_G\times  (j_{\mathbb L_{G/P}^\circ}\kappa))_\ast( \mathcal L'_\beta\boxtimes\mathcal L_\beta)\\
&\cong& R(\mathrm{id}_G\times  (j_{\mathbb L_{G/P}^\circ}\kappa))_\ast
((\mathrm{id}_G\times (j_{\mathbb L_{G/P}^\circ}\kappa))^\ast \pi_1^\ast \mathcal L'_\beta\otimes\pi'^\ast_2
\mathcal L_\beta)\\
&\cong&  \pi_1^\ast \mathcal L'_\beta\otimes R(\mathrm{id}_G\times  (j_{\mathbb L_{G/P}^\circ}\kappa))_\ast \pi'^\ast_2
\mathcal L_\beta\\
&\cong& \pi_1^\ast \mathcal L'_\beta\otimes \pi_2^\ast R( j_{\mathbb L_{G/P}^\circ}\kappa)_\ast \mathcal L_\beta.
\end{eqnarray*}
This proves the second assertion.
\end{proof}

\begin{lemma} \label{small} Let $S$ and $T$ be $k$-schemes of finite type, $E\to S$ and $F\to T$ 
vector bundles over $S$ and $T$ of ranks $r_1$ and $r_2$, respectively, 
$\phi:S\stackrel\cong \to T$ an isomorphism, and $g:E\to F\times_TS$ a morphism of vector bundles. 
Let $f:E\to F$ be the composite of
the morphism $g$ and the projection $\pi: F\times_TS\to F$. Let $f^\vee:F^\vee\to E^\vee$ be the transpose of 
$f$. Then
for any object $K\in\mathrm{ob}\, D_c^b(E,\overline{\mathbb Q}_\ell)$, we have 
$$\mathscr F_\psi(Rf_! K)\cong f^{\vee\ast} \mathscr F_\psi(K)[r_2-r_1].$$
\end{lemma}

Let's describe the transpose $f^\vee:F^\vee\to E^\vee$. We have commutative diagrams 
$$\begin{array}{ccccccc}
E&\stackrel f \to & F&& F^\vee&\stackrel{f^\vee}\to& E^\vee\\
\downarrow&&\downarrow&&\downarrow&&\downarrow\\
S&\stackrel\phi\to &T,&&T&\stackrel{\phi^{-1}}\to& S.
\end{array}$$
Let $\mathcal E$ and $\mathcal E^\vee$ be the $\mathcal O_S$-modules of sections of $E\to S$ and $E^\vee\to S$, respectively, 
and let  $\mathcal F$ and $\mathcal F^\vee$ be the $\mathcal O_T$-modules of sections of $F\to T$ and $F^\vee\to T$, respectively. 
The morphism $f:E\to F$ is completely characterized by the homomorphism of $\mathcal O_T$-modules
$f^\sharp: \phi_\ast \mathcal E \to\mathcal F.$
For any open subset $V$ of $T$, let $U=\phi^{-1}(V)$. The homomorphism 
$f^\sharp_V:\mathcal E(U)\to \mathcal F(V)$ maps a section $s:U\to E$ of the vector bundle $E\to S$ to the section 
$fs\phi^{-1}:V\to F$ of the vector bundle $F\to T$. Similarly, the morphism $f^\vee:F^\vee\to E^\vee$ is completely characterized 
by the homomorphism of $\mathcal O_S$-modules
$f^{\vee\sharp}: \phi^{-1}_\ast \mathcal F^\vee \to\mathcal E^\vee,$
and the homomorphism 
$f^{\vee\sharp}_U:\mathcal F^\vee(V)\to \mathcal E^\vee(U)$ maps a section $t^\ast:V\to F^\vee$ of the vector bundle $F^\vee\to T$ 
to the section $f^\vee t^\ast \phi:U\to E^\vee$ of the vector bundle $E^\vee\to S$.  For any $s\in\mathcal E(U)$ and $t^\ast\in\mathcal 
F^\vee(V)$, we have
$$\phi_V^\natural(\langle f^\sharp_V(s), t^\ast \rangle)=\langle s, f^{\vee\sharp}_U(t^\ast)\rangle,$$
where $\langle\,,\,\rangle$ denotes the pairings $\mathcal E\times\mathcal E^\vee\to \mathcal O_S$ and $\mathcal F\times \mathcal F^\vee
\to \mathcal O_T$, and $\phi^\natural: \mathcal O_T\to \phi_\ast \mathcal O_S$ is the morphism on structure sheaves coming from the 
morphism $\phi:S\to T$. 

\medskip
\begin{proof} [Proof of Lemma \ref{small}]Let $g^\vee:F^\vee\times_TS\to E^\vee$ be the transpose of $g$, and 
let $\pi^\vee:F^\vee\times_T S\to F^\vee$ be the projection. 
Note that $\pi^\vee$ is an isomorphism and $g^\vee=f^\vee\pi^\vee$. By \cite[1.2.3.5 and 1.2.2.4]{L}, we have 
\begin{eqnarray*}
\mathscr F_\psi(Rf_!K)&\cong& \mathscr F_\psi(R\pi_!Rg_!K)\\
&\cong& R\pi^\vee_! \mathscr F_\psi(Rg_!K)\\
&\cong& R\pi^\vee_! g^{\vee\ast}\mathscr F_\psi(K)[r_2-r_1]\\
&\cong& f^{\vee\ast}\mathscr F_\psi(K)[r_2-r_1].
\end{eqnarray*}
\end{proof}

\begin{lemma} \label{action} Let $X$ be an $k$-scheme of finite type, $E\to X$ a vector bundle, and $G$ an algebraic group 
over $k$ acting on $X$
and acting equivariantly and linearly on the vector bundle $E\to X$.  Let $m:G\times_k E\to E$ be the action of $G$ on $E$, and let 
$m^\vee: G\times_k E^\vee\to E^\vee$ be the 
contragredient action of $G$ on $E^\vee$. For any  $K\in\mathrm{ob}\, D_c^b(E,\overline{\mathbb Q}_\ell)$, we have 
$$m^{\vee\ast}\mathscr F_\psi(K)\cong \mathscr F_\psi(m^\ast K),$$
where the righthand side is the Deligne-Fourier transform for the vector bundle $G\times_k E\to G\times_k X$. 
\end{lemma}

\begin{proof} Let $f:G\times_k E\to G\times_k E$ be the isomorphism $(g,x)\mapsto (g, g^{-1}x)$, and let $\phi:G\times_k X\to G\times_k X$ be
the isomorphism defined in the same way. The transpose of $f$ as described in Lemma \ref{small} 
is exactly the morphism $$f^\vee: G\times_k E^\vee\to G\times_k E^\vee, \quad (g,x)\mapsto (g,gx).$$
Let  $\pi:G\times_k E\to E$ and $\pi^\vee: G\times_k E^\vee\to E^\vee$ be the projections.  Since  $m^\vee=\pi^\vee f^\vee$, we have
$$m^{\vee\ast} \mathscr F_\psi(K) \cong f^{\vee\ast}\pi^{\vee\ast} \mathscr F_\psi(K).$$
Since the Deligne-Fourier transform commutes with base change (\cite[1.2.2.9]{L}), we have 
$$\pi^{\vee\ast} \mathscr F_\psi(K)\cong \mathscr F_\psi(\pi^\ast K).$$ 
By Lemma \ref{small}, we have
$$f^{\vee\ast}\mathscr  F(\pi^\ast K)\cong  \mathscr F_\psi(Rf_! \pi^\ast K).$$ 
We thus have 
$$m^{\vee\ast} \mathscr F_\psi(K) \cong \mathscr F_\psi(Rf_! \pi^\ast K).$$ 
Since $f$ is an isomorphism and $\pi f^{-1}=m$, we have $$Rf_! \pi^\ast K\cong m^\ast K.$$
Our assertion follows. 
\end{proof}

\begin{lemma} \label{Kunneth} Let $S$ be an $k$-scheme of finite type, let $E\to S$ be a vector bundle, 
$Y\to S$ a morphism of finite type, and $K\in \mathrm{ob} \, D_c^b(E, \overline{\mathbb Q}_\ell)$ and 
$L\in \mathrm{ob} D_c^b(Y, \overline{\mathbb Q}_\ell)$.
Then we have $$\mathscr F_\psi(K\boxtimes L)\cong \mathscr F_\psi(K)\boxtimes L,$$ where on the lefthand side the Deligne-Fourier transform is for the vector bundle $E\times_SY\to Y$, and on the righthand side it is for the vector bundle $E\to S$.
\end{lemma}

\begin{proof} This follows from the definition of the Deligne-Fourier transform and the K\"unneth formula. 
\end{proof}

\begin{lemma}\label{tame} 

(i) Any multiplicative sheaf $\mathcal L$ on 
$\mathbb G_{m,k}$ is tamely ramified at $0$ and $\infty$. 

(ii) If a multiplicative sheaf on $\mathbb G_{m,k}$ is trivial on $\mathbb G_{m,\bar k}$, then it is trivial. 
\end{lemma}

\begin{proof} (i) This is proved in \cite[Lemma 3.5]{G}. For completeness, we include another proof, which 
is communicated to me by Haoyu Hu. By base change, 
we may assume $k$ is algebraically closed. Let 
$\bar\eta$ be a geometric generic point of $\mathbb G_m$, 
let $\rho: \pi_1(\mathbb G_m, \bar\eta)\to \overline{\mathbb Q}_\ell^*$ be the character defined by $\mathcal L$, and 
let $P$ be the wild inertia subgroup at $0$ or at $\infty$. The group $P$ is a pro-$p$-group and $\rho(P)$ is finite. 
Let $p^m$ be the number of elements in $\rho(P)$. Then we have $\rho^{p^m}|_{P}=1$. 
Let $$m_n: \mathbb G_{m, k}^n\to \mathbb G_{m, k}$$ be the morphism defined by multiplication. We have 
$$m_n^*\mathcal L\cong \boxtimes^n \mathcal L.$$ 
Pulling back this isomorphism by the diagonal morphism $\mathbb G_{m,k}\to  \mathbb G_{m, k}^n$, 
we get 
$$[n]^*\mathcal L\cong \mathcal L^{\otimes n},$$
where $[n]:\mathbb G_{m,k}\to \mathbb G_{m, k}$ is the morphism $x\mapsto x^n$. 
In particular, the character corresponding to the sheaf $[p^m]^*\mathcal L$ is $\rho^{p^m}$. Since $\rho^{p^m}|_{P}=1$, 
$[p^m]^*\mathcal L$  is tamely ramified at $0$ and $\infty$. But $[p^m]$ is a finite surjective radiciel morphism, and it has 
no effect on \'etale topology. So $\mathcal L$ is tame at $0$ and $\infty$. 

(ii) Let $\mathcal L$ be a multiplicative sheaf on  $\mathbb G_{m,k}$ which is trivial on  $\mathbb G_{m,\bar k}$. Then $\mathcal L$
is isomorphic to the inverse image of a sheaf on $\mathrm{Spec}\,k$. But the restriction of $\mathcal L$ to the unit section $1:\mathrm{Spec}\, k\to 
\mathbb G_{m, k}$
is trivial. So we have $\mathcal L\cong\overline{\mathbb Q}_\ell$. 
\end{proof}

\begin{lemma} \label{basechange} $Rj_{\mathbb L_{G/P}^\circ, \ast}\mathcal L_\beta$ commutes with any base change 
$Y\to G/P$. 
\end{lemma}

\begin{proof} Note that $\mathbb L^\circ_{G/P}\to G/P$ is a principal homogeneous space 
for the multiplicative group-scheme $\mathbb G_{m, G/P}$ over $G/P$, and
$\mathbb L_{G/P}\to G/P$ is 
the associated line bundle for the canonical action of $\mathbb G_m$ on $\mathbb A^1$. Let $\{U_i\}$ be a Zariski open covering of $G/P$
such that the 
line bundle $\mathbb L_{G/P}$ is trivial over each $U_i$, and we compactify $\mathbb L_{U_i}^\circ$ to a $\mathbb P^1$-bundle 
by adding the $0$ section and the $\infty$ section.  
The sheaf  $\mathcal L_\beta$
is homogeneous on $\mathbb L^\circ_{G/P}$ with respect to the $\mathbb G_{m, G/P}$-action. 
In particular, when restricted to the generic point $\eta$ of $G/P$, $\mathcal L_\beta|_{\mathbb L^\circ_\eta}$ is a multiplicative sheaf on 
$\mathbb G_{m,\eta}$. By Lemma \ref{tame}, $\mathcal L_\beta|_{\mathbb L^\circ_\eta}$ is tamely ramified at $0$ and
$\infty$. This implies that $\mathcal L_\beta|_{\mathbb L_{U_i}^\circ}$ is tamely ramified
at the $0$ section and the $\infty$ section. Our assertion then follows from \cite[XIII 2.1.10]{SGA7}.  
\end{proof}

Let $G_!(\beta,\psi)$ and $G_*(\beta, \psi)$ be the sheaf on $\mathrm{Spec}\,k$ defined 
in Theorem \ref{mainthm}.

\begin{lemma} \label{fiber}  Regard $v^\ast$ as a $k$-point in the fiber of ${\mathbb L}^{\vee}_{G/P,eP}\cong 
{\mathbb L}^\vee_{[v]}$, and 
let $0_{[v]}$ be the zero element in the line ${\mathbb L}^{\vee}_{G/P,eP}\cong {\mathbb L}^{\vee}_{[v]}$. We have 
\begin{eqnarray*}
\mathscr F_\psi(j_{\mathbb L_{G/P}^\circ, !} \mathcal L_\beta[1])|_{v^\ast}&\cong& G_!(\beta,\psi)[1], \\
\mathscr F_\psi(Rj_{\mathbb L_{G/P}^\circ, \ast} \mathcal L_\beta[1])|_{v^\ast}&\cong& G_*(\beta,\psi)[1],\\
\mathscr F_\psi(Rj_{\mathbb L_{G/P}^\circ, \ast} \mathcal L_\beta[1])|_{0_{[v]}}&\cong& 0.
\end{eqnarray*} 
Moreover $G_!(\beta,\psi)$ and $G_*(\beta,\psi)$ are of rank $1$. 
\end{lemma}

\begin{proof} With respect to the morphism
$\mathrm{Spec}\,k\to {G/P}$ defined by the rational point $eP$, 
the base change of the morphism $j_{\mathbb L_{G/P}^\circ} :\mathbb L_{G/P}^\circ\to \mathbb L_{G/P}$  
can be identified the canonical open immersion $j: \mathbb G_{m,k}\hookrightarrow \mathbb A^1_k$. 
By Lemma \ref{basechange}, the base change of 
$Rj_{\mathbb L_{G/P}^\circ, \ast} \mathcal L_\beta[1]$ can be identified with $Rj_\ast (\mathcal L_\beta|_{\mathbb G_{m,k}})[1]$. 
The Deligne-Fourier transform and $j_{\mathbb L_{G/P}^\circ, !} \mathcal L_\beta[1]$ also commute with base change (\cite[1.2.2.9]{L}). 
So we are reduced to show 
\begin{eqnarray*}
\mathscr F_\psi(j_!(\mathcal L_\beta|_{\mathbb G_{m,k}})[1])|_{1}&\cong& G_*(\beta,\psi)[1],\\
\mathscr F_\psi(Rj_\ast (\mathcal L_\beta|_{\mathbb G_{m,k}})[1])|_{1}&\cong& G_!(\beta,\psi)[1],\\
\mathscr F_\psi(Rj_\ast (\mathcal L_\beta|_{\mathbb G_{m,k}})[1])|_{0}&\cong& 0,
\end{eqnarray*}
where the lefthand side are restrictions to the unit section $1:\mathrm{Spec}\,k\to \mathbb A^1_k$ and the zero section
$0:\mathrm{Spec}\,k\to \mathbb A^1_k$. 

For convenience, denote $\mathcal L_\beta|_{\mathbb G_{m,k}}$ by $\mathcal L$. 
First consider the case where $\mathcal L|_{\mathbb G_{m,\bar k}}$ is nontrivial. In this case, we claim that 
$(Rj_\ast \mathcal L)_{\bar 0}=0$ so that we have $j_!\mathcal L\cong Rj_*\mathcal L$. 
Indeed, by Lemma \ref{tame}, $\mathcal L|_{\mathbb G_{m,\bar k}}$ is a nontrivial multiplicative sheaf 
on $\mathbb G_{m,\bar k}$. It is necessarily isomorphic to a nontrivial Kummer sheaf $\mathcal K_\chi$ for a character
$\chi: \varprojlim_{(n,p)=1} \mu_n(\bar k) \to\overline{\mathbb Q}_\ell^*$. The invariant and coinvariant of the inertia subgroup $I_0$ at $0$ for the 
representation of $\pi_1(\mathbb G_{m,\bar k})$ corresponding to $\mathcal K_\chi$ are trivial. Our claim follows
the calculation in \cite[Dualit\'e 1.3]{SGA4.5}. So it suffices to prove 
$$
\mathscr F_\psi(j_!\mathcal L[1])|_{1}\cong G_!(\beta,\psi)[1], \quad
\mathscr F_\psi(j_! \mathcal L[1])|_{0}\cong 0.$$
By
definition, we have 
$$\mathscr F_\psi(j_!\mathcal L[1])_{\bar 1}\cong R\Gamma_c(\mathbb G_{m,\bar k}, \mathcal L\otimes\mathcal L_\psi)[2].$$ 
We have 
$$H_c^0(\mathbb G_{m,\bar k}, \mathcal L\otimes\mathcal L_\psi)=0$$ since $\mathbb G_{m,\bar k}$ is an affine curve. 
We have 
$$H_c^2(\mathbb G_{m,\bar k}, \mathcal L\otimes\mathcal L_\psi)\cong 
H^0(\mathbb G_{m,\bar k}, \mathcal L^{-1}\otimes\mathcal L_{\psi^{-1}})(-1)=0$$
by Poincar\'e duality and the fact that the invariant of the 
representation of $\pi_1(\mathbb G_{m,\bar k})$ corresponding to $\mathcal L\otimes \mathcal L_\psi$ is trivial.
By the Grothendieck-Ogg-Shafarevich formula and the fact that $\mathcal L\otimes \mathcal L_\psi$ is lisse on 
$\mathbb G_{m,\bar k}$, tame at $0$ and with Swan conductor $1$ at $\infty$, we have
$$\chi_c(\mathbb G_{m,\bar k}, \mathcal L\otimes\mathcal L_\psi)=-1.$$
So $H^1_c(\mathbb G_{m,\bar k},\mathcal L\otimes \mathcal L_\psi)$ is of rank $1$. Hence 
$$\mathscr F_\psi(j_!(\mathcal L_\beta|_{\mathbb G_{m,k}})[1])_{1}\cong G_!(\beta,\psi)[1],$$ 
and $G_*(\beta,\psi)$
is of rank 1. Similarly one can prove $H^1_c(\mathbb G_{m,\bar k},\mathcal L)=0$ for all $i$. 
Hence $R\Gamma_c(\mathbb G_{m,\bar k}, \mathcal L)=0$. So we have 
$$\mathscr F_\psi(j_!\mathcal L[1])_{\bar 0}\cong R\Gamma_c(\mathbb G_{m,\bar k}, \mathcal L)[2]=0.$$ 

Next, suppose $\mathcal L_{\mathbb G_{m,\bar k}}$ is trivial. By Lemma \ref{tame} (ii), we have 
$\mathcal L\cong\overline{\mathbb Q}_\ell$. So 
we have 
$$\mathscr F_\psi(j_!\mathcal L[1])_{\bar 1}\cong 
R\Gamma_c(\mathbb G_{m,\bar k}, \mathcal L_\psi)[2].$$ We have a distinguished triangle 
$$R\Gamma_c(\mathbb G_{m,\bar k}, \mathcal L_\psi)\to R\Gamma_c(\mathbb A^1_{\bar k},\mathcal L_\psi)
\to \mathcal L_{\psi,\bar 0}\to.$$
By \cite[Sommes trig. 2,7]{SGA4.5}, we have 
$$R\Gamma_c(\mathbb A^1_{\bar k}, \mathcal L_\psi)=0,\quad \mathcal L_\psi|_0\cong\overline{\mathbb Q}_\ell.$$ 
So we have $$R\Gamma_c(\mathbb G_{m,\bar k}, \mathcal L_\psi)[2]\cong \overline{\mathbb Q}_\ell[1].$$ Therefore 
$$\mathscr F_\psi(j_!\mathcal L[1])_{1}\cong \overline{\mathbb Q}_\ell[1]\cong G_!(\beta,\psi)[1].$$
Denote by $D$ the Verdier dual functor on $\mathbb A^1$. 
We have 
\begin{eqnarray*}
\mathscr F_\psi(Rj_*\overline{\mathbb Q}_\ell[1])&\cong& 
\mathscr F_\psi(D(j_!\overline{\mathbb Q}_\ell[1]))(-1)\\
&\cong& D(\mathscr F_{\psi^{-1}}(j_!\overline{\mathbb Q}_\ell[1]))(-2).    
\end{eqnarray*}
By \cite[2.3.1.3]{L}, we have $$\mathscr F_{\psi^{-1}}(j_!\overline{\mathbb Q}_\ell[1])|_{\mathbb A^1_k-\{0\}}\cong
\mathcal F[1]$$ for a lisse sheaf $\mathcal F$ on $\mathbb A_k^1-\{0\}$, and by the formula that we have already
proved, we have $\mathcal F|_{1}\cong \overline{\mathbb Q}_\ell$. 
So 
$$\mathscr F_\psi(Rj_*\mathcal L[1])_{1}\cong \mathscr F_\psi(Rj_*\overline{\mathbb Q}_\ell[1])_{1}\cong
\mathrm{Hom}(\mathcal F_{\bar 1}, \overline{\mathbb Q}_\ell)[1](-1)\cong  
\overline{\mathbb Q}_\ell(-1)[1]=G_*(\beta,\psi)[1].$$
We have 
\begin{eqnarray*}
\mathscr F_\psi(Rj_*\mathcal L[1])_{\bar 0}\cong \mathscr F_\psi(Rj_*\overline{\mathbb Q}_\ell)_{\bar 0}
\cong R\Gamma_c(\mathbb A^1_{\bar k}, Rj_*\overline{\mathbb Q}_\ell).
\end{eqnarray*}
Let $J: \mathbb G_{m,k}\hookrightarrow \mathbb P^1_k$
be the open immersion. We have a distinguished triangle
$$\begin{array}{ccccc}
R\Gamma_c(\mathbb A^1_{\bar k}, Rj_*\overline{\mathbb Q}_\ell)\to& 
R\Gamma (\mathbb P^1_{\bar k}, RJ_*\overline{\mathbb Q}_\ell)&\to& R\Gamma(\overline\infty, RJ_*\overline{\mathbb Q}_\ell)&\to \\
&\wr\!\!\parallel&&\wr\!\!\parallel&\\
&R\Gamma (\mathbb G_{m,\bar k}, \overline{\mathbb Q}_\ell)&&(RJ_*\overline{\mathbb Q}_\ell)_{\overline\infty}&
\end{array}$$ 
We claim that 
$$R\Gamma (\mathbb G_{m,\bar k}, \overline{\mathbb Q}_\ell)\cong (RJ_*\overline{\mathbb Q}_\ell)_{\overline\infty}.$$
Therefore $R\Gamma_c(\mathbb A^1_{\bar k}, Rj_*\overline{\mathbb Q}_\ell)=0$ and hence $\mathscr F_\psi(Rj_*\mathcal L[1])_{0}=0.$
To prove the claim, let $\tilde K_\infty$ be the fraction 
field of the strict henselization of $\mathbb G_m$ at $\infty$. We have 
$$(RJ_*\overline{\mathbb Q}_\ell)_{\overline\infty}\cong R\Gamma(\mathrm{Spec}\, \tilde K_\infty, \overline{\mathbb Q}_\ell).$$
It suffices to show 
$$H^i(\mathbb G_{m,\bar k},\mu_n)\cong H^i(\mathrm{Spec}\,
\tilde K_\infty, \mu_n)$$ for all $(n, p)=1$. We compute $H^i(\mathbb G_{m,\bar k},\mu_n)$ and $H^i(\mathrm{Spec}\,
\tilde K_\infty, \mu_n)$ and compare them using the long exact sequences of cohomology groups associated to the
the Kummer short exact sequence.
\end{proof}

\begin{lemma} \label{mainlemma} We have 
\begin{eqnarray*}
\mathscr F_{\psi,\mathbb L_{G/P}}
(j_{\mathbb L_{G/P}^\circ, !}\mathcal L_\beta[1])&\cong& R j_{\mathbb L_{G/P}^{\vee\circ}, \ast}  \mathcal L_\beta[1]
\otimes G_!(\beta,\psi), \\
\mathscr F_{\psi,\mathbb L_{G/P}}
(Rj_{\mathbb L_{G/P}^\circ, \ast}R \mathcal L_\beta[1])&\cong&  j_{\mathbb L_{G/P}^{\vee\circ}, !} \mathcal L_\beta[1]\otimes
G_*(\beta,\psi).
\end{eqnarray*}
\end{lemma}

\begin{proof} The first assertion can be deduced from the second one by taking Verdier dual. 
Let's prove the second assertion. 

Let $m:G\times_k \mathbb L_{G/P}\to \mathbb L_{G/P}$ be the action of $G$ on $\mathbb L_{G/P}$, let 
$m^\vee:G\times_k \mathbb L_{G/P}^\vee\to 
\mathbb L_{G/P}^\vee$ be the contragredient action, and let $\phi$ be the composite 
$$G\to G/Q\stackrel{\kappa^\vee}{\underset\cong\to} \mathbb L_{G/P}^{\vee\circ}.$$ 
By Lemmas \ref{equivariant}, \ref{action}, and \ref{Kunneth}, we have 
\begin{eqnarray*}
m^{\vee\ast} \mathscr F_\psi
(Rj_{\mathbb L_{G/P}^\circ, \ast} \mathcal L_\beta[1])&\cong& 
\mathscr F_\psi(m^\ast Rj_{\mathbb L_{G/P}^\circ, \ast} \mathcal L_\beta[1])\\
&\cong& 
\mathscr F_\psi (\mathcal L'_\beta\boxtimes Rj_{\mathbb L_{G/P}^\circ, \ast}\mathcal L_\beta[1])\\
&\cong& \mathcal L'_\beta \boxtimes \mathscr F_\psi(Rj_{\mathbb L_{G/P}^\circ, \ast} \mathcal L_\beta[1])
\\
&\cong& \phi^\ast\mathcal L_\beta \boxtimes \mathscr F_\psi(Rj_{\mathbb L_{G/P}^\circ, \ast} \mathcal L_\beta[1]).
\end{eqnarray*}
Regard $v^\ast$ as a point in the fiber of ${\mathbb L}^{\vee}_{G/P,eP}\cong 
({\mathbb L}^\vee)_{[v]}$. Restricting the isomorphism $$m^{\vee\ast} \mathscr F_\psi
(Rj_{\mathbb L_{G/P}^\circ, \ast}\mathcal L_\beta[1])\cong 
\phi^\ast\mathcal L_\beta \boxtimes \mathscr F_\psi(Rj_{\mathbb L_{G/P}^\circ, \ast} \mathcal L_\beta[1])$$ to $G\times v^\ast$, we get
$$\phi^\ast j^\ast_{\mathbb L_{G/P}^{\vee\circ}}\mathscr F_\psi
(Rj_{\mathbb L_{G/P}^\circ, \ast} \mathcal L_\beta[1])
\cong \phi^\ast \mathcal L_\beta \otimes \mathscr F_\psi(Rj_{\mathbb L_{G/P}^\circ, \ast} \mathcal L_\beta[1])|_{v^\ast}.$$
By Lemma \ref{fiber}, we have
$$\mathscr F_\psi(Rj_{\mathbb L_{G/P}^\circ, \ast} \mathcal L_\beta[1])|_{v^\ast}\cong G_*(\beta,\psi)[1].$$
We thus have 
$$\phi^\ast j^\ast_{\mathbb L_{G/P}^{\vee\circ}}\mathscr F_\psi
(Rj_{\mathbb L_{G/P}^\circ, \ast} \mathcal L_\beta[1])
\cong \phi^\ast(\mathcal L_\beta[1] \otimes G_*(\beta,\psi)).$$
The morphism $\phi$ is smooth. 
Locally with respect to \'etale topology, $\phi$ has sections. The above isomorphism then implies that 
$j^\ast_{\mathbb L_{G/P}^{\vee\circ}}\mathscr F_\psi
(Rj_{\mathbb L_{G/P}^\circ, \ast}\mathcal L_\beta)$ is a lisse sheaf on $\mathbb L_{G/P}^{\vee\circ}$. The fibers of
$\phi$ are geometrically connected. This implies that if $V\to \mathbb L_{G/P}^{\vee\circ}$ is an \'etale covering space such that $V$ is connected, 
then $G\times_{\mathbb L_{G/P}^{\vee\circ}}V$ is also connected. So the homomorphism induced by $\phi$ on fundamental groups 
$\pi_1(G)\to \pi_1(\mathbb L_{G/P}^{\vee\circ})$ is surjective. The 
isomorphism $\phi^\ast j^\ast_{\mathbb L_{G/P}^{\vee\circ}}\mathscr F_\psi
(Rj_{\mathbb L_{G/P}^\circ, \ast}\mathcal L_\beta[1])
\cong \phi^\ast(\mathcal L_\beta[1] \otimes G_*(\beta,\psi))$
thus implies that 
$$j^\ast_{\mathbb L_{G/P}^{\vee\circ}}\mathscr F_\psi
(Rj_{\mathbb L_{G/P}^\circ, \ast}\mathcal L_\beta[1])
\cong \mathcal L_\beta[1] \otimes G_*(\beta,\psi).$$
By Lemma
\ref{fiber}, we have 
$$\mathscr F_\psi(Rj_{\mathbb L_{G/P}^\circ, \ast}\mathcal L_\beta[1])|_{0_{[v]}}=0.$$
Restricting the isomorphism $m^{\vee\ast} \mathscr F_\psi
(Rj_{\mathbb L_{G/P}^\circ, \ast} \mathcal L_\beta[1])\cong 
\phi^\ast\mathcal L_\beta \boxtimes \mathscr F_\psi(Rj_{\mathbb L_{G/P}^\circ, \ast}\mathcal L_\beta[1])$ to $G\times 0_{[v]}$,
we see the restriction of $\mathscr F_\psi
(Rj_{\mathbb L_{G/P}^\circ, \ast} \mathcal L_\beta[1])$ to the zero section of 
${\mathbb L}_{G/P}^{\vee}$ vanishes. It follows that $$\mathscr F_\psi
(Rj_{\mathbb L_{G/P}^\circ, \ast} \mathcal L_\beta[1])
\cong j_{\mathbb L_{G/P}^{\vee\circ}, !}\mathcal L_\beta[1] \otimes G_*(\beta,\psi).$$
This proves our assertion.
\end{proof}

\section{Tautological systems associated to a family of representations}

Let $G$ be an algebraic group $G$ over a perfect field $k$, let
$\rho_\lambda: G\to \mathbb \mathrm{GL}(V_\lambda)$ ($\lambda\in \Lambda)$ be a finite family of representations 
so that the morphism $$\iota: G\to \prod_{\lambda\in \Lambda}\mathrm{End}(V_\lambda), 
\quad g\mapsto (\rho_\lambda(g))$$ is quasi-finite, and let
$$V=\prod_{\lambda\in \Lambda}\mathrm{End}(V_\lambda).$$
We have an action of $G$ on $V$ given by 
$$(g, (A_\lambda)_{\lambda\in\Lambda})\mapsto (\rho_\lambda(g)A_\lambda)_{\lambda\in\Lambda}$$ 
for any points $g$ in $G$ and $(A_\lambda)_{\lambda\in \Lambda}$ in $V$. Take $v$ to be the $k$-point $v=(\mathrm{id}_{V_\lambda})_{\lambda\in\Lambda}$
of $V$. 
Then the connected component of the stabilizer of $v$ is trivial. 
We have a perfect pairing 
$$\langle\;,\;\rangle:
V\times V\to k, \quad ((A_\lambda)_{\lambda\in\Lambda}, (B_\lambda)_{\lambda\in\Lambda})\mapsto \sum_{\lambda\in\Lambda}
\mathrm{Tr}(A_\lambda B_\lambda).$$ 
We identify $V$ with $V^\vee$ through this pairing. 
Let $\mathcal L_\beta$ be a multiplicative sheaf on the maximal abelian quotient $G/[G,G]$ of $G$. By our assumption, $\iota$ is 
both quasi-finite and affine. So $\iota_! \mathcal L_\beta[n]$ and $R\iota_* \mathcal L_\beta[n]$ are perverse sheaves. Hence 
the $\ell$-adic tautological systems $\mathcal T_!(G, \beta, V, v,\psi)=\mathscr F_\psi(\iota_!\mathcal L_\beta[n])$ 
and $\mathcal T_*(G, \beta, V, v,\psi)=\mathscr F_\psi(R\iota_*\mathcal L_\beta[n])$ on $V$ are perverse.  

\begin{lemma} Notation as above. Let $F: G \times_k V \to \mathbb{A}_k^{1}$ be the morphism defined by 
$$F(g,(A_\lambda)_{\lambda\in\Lambda})=\sum_{\lambda\in\Lambda}\mathrm{Tr}(\rho_\lambda(g)A_\lambda),$$ 
and let $$\pi_1: G\times_k V\to G, \quad \pi_2: G\times_k V\to V$$ be the projections.
Then we have  
$$\mathcal T_!(G, \beta, V, v,\psi)\cong 
R\pi_{2!} (\pi_{1}^{\ast}\mathcal L_\beta \otimes F^{\ast}L_{\psi})[n+N],$$ 
where $n=\mathrm{dim}\, G$ and $N=\mathrm{dim}\,V$. 
\end{lemma}

\begin{proof} We have a commutative diagram
$$\begin{array}{rclcc}
G \times_k V&\stackrel{\iota\times{\mathrm{id_V}}}\to& V\times_k V&\stackrel{\langle\;,\;\rangle}\to& \mathbb A^1\\
{\scriptstyle \pi_1}\downarrow&&\downarrow{\scriptstyle \mathrm{pr}_1}&&\\
G&\stackrel{\iota}\to & V.&&
\end{array}$$
By the proper base change theorem and the projection formula, we have
\begin{eqnarray*}
\mathcal T_!(G, \beta, V, v,\psi)
 &\cong & R\mathrm{pr}_{2!}\Big(\mathrm{pr}_{1}^{\ast}\iota_{!}\mathcal L_\beta \otimes \langle\;,\;\rangle^{\ast}\mathcal L_{\psi}\Big) [n+N] \\
  &\cong& R \mathrm{pr}_{2!}\Big((\iota\times \mathrm{id}_V)_{!}\pi_{1}^{\ast}\mathcal L_\beta \otimes  \langle\;,\;\rangle^{\ast}\mathcal L_{\psi}\Big) [n+N]\\
  &\cong&R\big( \mathrm{pr}_{2}\circ (\iota\times\mathrm{id}_V)\big)_{!}\Big(\pi_{1}^{\ast}\mathcal L_\beta \otimes (\langle\;,\;\rangle\circ 
  \big(\iota \times \mathrm{id}_V)\big)^{\ast}\mathcal L_{\psi}\Big)[n+N] \\
  &\cong&  R\pi_{2!} (\pi_{1}^{\ast}\mathcal L_{\beta} \otimes F^{\ast}\mathcal L_{\psi})[n+N].
\end{eqnarray*}
\end{proof}

We call $$\mathrm{Hyp}_{\psi}(\Lambda, \beta)=
R\pi_{2!} (\pi_{1}^{\ast}\mathcal L_\beta \otimes F^{\ast}L_{\psi})[n+N]$$ 
the \emph{hypergeometric sheaf} on $V$. 
Over the complex number field $\mathbb C$ and in the case where $G$ is a reductive group and $\Lambda$ is a finite family 
of irreducible representations of $G$, this is introduced in \cite{K}. The GKZ hypergeometric 
sheaf in Example \ref{GKZ} is a special case for the group $G=\mathbb G_m^n$ and the family of representations defined by the 
character 
$$\mathbb G_m^n\to \mathbb G_m, \quad (t_1, \ldots, t_n)\mapsto t_1^{w_{1j}}\cdots t_n^{w_{nj}}\quad (j=1, \ldots, N).$$
Suppose $k=\mathbb F_q$ is a finite field, and suppose $\mathcal L_\beta$ 
is the Lang sheaf associated to a character 
$$\beta: (G/[G,G])(\mathbb F_q)\to \overline{\mathbb Q}_\ell^*.$$
Then for any $\mathbb F_q$-point $A=(A_{\lambda})_{\lambda\in\Lambda}$ of $V$, we have 
$$\mathrm{Tr}(\mathrm{Frob}_A, \mathrm{Hyp}_{\psi}(\beta, \Lambda)_{\bar A})=(-1)^{n+N}
\sum_{g\in G(\mathbb F_q)} \beta(g) \psi \Big( \mathrm{Tr}_{\mathbb F_q/\mathbb F_p}\Big(\sum_{\lambda\in \Lambda} 
\mathrm{Tr}(\rho_\lambda(g)A_\lambda)\Big)\Big)$$
by the Grothendieck trace formula. 

\medskip
\noindent{\bf Example 1}. Let $G\subset \mathrm{GL}_m$ and let $\Lambda$ be the set consisting of only the standard representation
of $\rho: G\hookrightarrow \mathrm{GL}_m$. Determinant on $\mathrm{GL}_m$ induces a homomorphism 
$$G/[G,G]\to \mathbb G_m.$$ 
Suppose $k=\mathbb F_q$ is a finite field, and let $\beta:\mathbb G_m(\mathbb F_q)\to\overline{\mathbb Q}_\ell^*$ be a multiplicative character.  
For any rational point $A=(a_{ij})$ in $\mathfrak{gl}_m$, where $(a_{ij})$ is an $m\times m$ 
matrix with entries in $\mathbb F_q$, we have 
$$\mathrm{Tr}(\mathrm{Frob}_A, \mathrm{Hyp}_{\psi}(\beta,\Lambda)_{\bar A})=
\sum_{(g_{ij})\in G(\mathbb F_q)} \beta(\mathrm{det}\big(g_{ij})\big) \psi \Big( \mathrm{Tr}_{\mathbb F_q/\mathbb F_p}(\sum_{i,j}^m
g_{ij} a_{ji})\Big).$$

\noindent{\bf Example 2} (\cite[Example 6.1]{K}). 
Let $L$ and $N$ be two 1-dimensional $k$-vector spaces, and let $V$ be an $n$-dimensional
$k$-vector space. Set
$$G = \mathrm{GL}(N) \times_k \mathrm{GL}(L) \times_k \mathrm{GL}(V)
 \cong \mathbb{G}_m\times_k\mathbb G_m \times_k \mathrm{GL}_{n}.$$
Let $\Lambda=\{N\otimes_k N, N \otimes_k L, N \otimes_k V, L\otimes_k V\}$. In this
case, we have 
$$V = \mathrm{End}(N\otimes_k N)
  \oplus \mathrm{End}(N\otimes_k L)
  \oplus \mathrm{End}(N\otimes_k V)
  \oplus \mathrm{End}(L\otimes_k V)$$ which is of dimension $2(1+n^2)$. 
The stabilizer in $G$ of the point  $v=(\mathrm{id}_{N\otimes_k N}, \mathrm{id}_{N\otimes_k L}, \mathrm{id}_{N\otimes_k V}, 
\mathrm{id}_{L\otimes_k V})$ in $V$ is trivial. Let $\iota: G\to V$ be the morphism
$$\iota: \mathbb G_m\times_k\mathbb G_m\times_k\mathrm{GL}_n\to \mathbb A^1\times_k\mathbb A^1\times_k\mathfrak {gl}_n
\times_k\mathfrak {gl}_n, \quad 
(s, t, g)\mapsto (s^2, st, sg, tg).$$ Its image can be identified with the orbit of $v$. 
Note that $\bar k$-points in the image of $\iota$ can be described by 
$$\{(u, v, X, Y)\in \bar k\times \bar k\times \mathfrak {gl}_n(\bar k)\times 
\mathfrak {gl}_n(\bar k):\; 
 uY - vX = 0\}.$$
We have $$G/[G, G]\cong  \mathbb{G}_m\times_k\mathbb G_m \times_k \mathrm{GL}_{n}/[\mathrm{GL}_{n}, \mathrm{GL}_{n}]\cong
\mathbb{G}_m\times_k\mathbb G_m \times_k\mathbb{G}_m.$$
Suppose $k=\mathbb F_q$ is a finite field with $q$ elements. Let $\chi_1, \chi_2, \chi_3: \mathbb F_q^*\to \overline{\mathbb Q}_\ell^*$ 
be multiplicative characters and let $\beta: (G/[G, G])(\mathbb F_q)\to \overline{\mathbb Q}_\ell^*$ be the character
$$(s, t, g)\mapsto \chi_1(s)\chi_2(g)\chi_3(\mathrm{det}(g)).$$ For any $\mathbb F_q$-points $x=(a, b, C, D)$ of $V$, we have 
$$\mathrm{Tr}(\mathrm{Frob}_x, (\tau_!(G, \beta, V, v,\psi)_{\bar x})=\sum_{s\in\mathbb F_q^*, \,t\in\mathbb F_q^*,\, g\in\mathrm{GL}_n(\mathbb F_q)}
\chi_1(s)\chi_2(t)\chi_3(\mathrm{det}(g)) \psi\Big(s^2a+st b+ s\mathrm{Tr}(gC)+t\mathrm{Tr}(gD)\Big).$$


\begin{thebibliography}{22}

\bibitem{D} P. Deligne, La conjecture de Weil II, {\it Publ. Math.
IHES.}, 52 (1981), 313-428.

\bibitem{SGA4.5} P. Deligne, Cohomologie
\'etale (SGA $4\frac{1}{2}$), {\it Lecture Notes in Math.} 569,
Springer-Verlag (1977).

\bibitem{SGA7} P. Deligne and N. Katz, Groupes de Mondoromie en G\'eom\'etrie Alg\'ebrique (SGA 7), \emph{Lecture
Notes in Math.} 340, Springer-Verlag (1973).  

\bibitem{FGKZ} L. Fu, Gelfand-Kapranov-Zelevinsky hypergeometric sheaves, 
\emph{Proceedings of the 6th International Congress of Chinese Mathematicians}
Vol. I, 281-295, Adv. Lect. Math. (ALM), 36, Int. Press, Somerville, MA, 2017.

\bibitem{F} L. Fu, $\ell$-adic Gelfand-Kapranov-Zelevinsky sheaves, \emph{Adv. in Math.}  298 (2016), 51-88. 

\bibitem{G} Q. Guignard, On the ramified class field theory of relative curves, \emph{Algebra and Number Theory}
13-6 (2019), 1299-1326.

\bibitem{GKZ} I. M. Gelfand, A. V. Zelevinsky, M. M. Kapranov,
Hypergeometric functions and toric varieties, English translation,
{\it Functional Anal. Appl.} 23 (1989), 94-106; Correction to the
paper ``Hypergeometric functions and toric varieties'', English
translation, {\it Functional. Anal. Appl.} 27 (1995), 295.

\bibitem{H} R. Hotta, Equivariant $D$-modules, arXiv: math/9805021v1 (1998).

\bibitem{HLZ} A. Huang, B. H. Lian, X. Zhu, Period integrals and the Riemann-Hilbert correspondence.
\emph{J. Differential Geom.} 104 (2016), 325-369.

\bibitem{K} M. M. Kapranov, Hypergeometric functions on reductive groups, in \emph
{Integrable Systems and Algebraic Geometry (Kobe/Kyota 1997)} 236-281, World Sci. Publ. 1998.

\bibitem{KL} N. Katz, G. Laumon, Transformation de Fourier et
majoration de sommes exponentielles, {\it Publ. Math. IHES} 62
(1985), 361-418.

\bibitem{L} G. Laumon,  Transformation de Fourier, constantes
d'\'equations fontionnelles, et conjecture de Weil, {\it Publ. Math.
IHES} 65 (1987), 131-210.

\bibitem{LY} B. H. Lian, S.-T. Yau, Period integrals of CY and general type complete intersections, {\it Invent .Math.}
191 (2013) 35-89.

\bibitem{MB} L. Moret-Bailly, Pinceaux de Vari\'et\'es de Ab\'eliennes, \emph{Ast\'erique} 129 (1985). 

\bibitem{Serre} J.-P. Serre, \emph{Algebraic Groups and Class Fields}, Springer-Verlag 1988.

\bibitem{S} T. A. Springer, \emph{Linear Algebraic Groups}, 2nd edition, Birkh\"auser 1998. 


\end{thebibliography}
\end{document}